\documentclass[12pt]{amsart}
\newtheorem{theorem}{Theorem}[section]
\newtheorem{lemma}[theorem]{Lemma}

\newtheorem{proposition}[theorem]{Proposition}
\theoremstyle{definition}
\newtheorem{definition}[theorem]{Definition}
\newtheorem{example}[theorem]{Example}

\usepackage{enumerate}
\newtheorem{remark}[theorem]{Remark}
\usepackage[margin=0.75in]{geometry}
\usepackage{calligra}
\usepackage{bbm}
\usepackage{xcolor, soul}
\sethlcolor{yellow}
\usepackage{mathtools}

\DeclarePairedDelimiter\floor{\lfloor}{\rfloor}
\DeclareMathOperator{\sinc}{sinc}
\DeclareMathOperator{\Ci}{Ci}
\DeclareMathOperator{\Si}{Si}
\DeclareMathOperator{\Ei}{Ei}
\DeclareMathOperator\supp{supp}
\DeclareMathOperator\tr{tr}
\DeclareMathOperator{\re}{Re}

\numberwithin{equation}{section}

\def\N{{\mathbb{N}}}
\def\h{{\mathcal{H}}}
\def\R{{\mathbb{R}}}

\def\b{{\mathcal{B}_2}}
\def\Z{{\mathbb{Z}^n}}
\def\z{{\mathbb{Z}}}

\def\c{{\mathbb{C}}}
\def\i{{\mathcal{I}}}

\def\t{{T_{(k,l)}^t}}
\def\p{{\phi}}
\def\mra{{multiresolution\ analysis}}
\def\x{{\chi}}
\newcommand{\abs}[1]{\left\lvert#1\right\rvert}
\newcommand\numberthis{\addtocounter{equation}{1}\tag{\theequation}}

\DeclareMathOperator{\Span}{span}
\newcommand{\be}{\begin{equation}}
\newcommand{\ee}{\end{equation}}
\newcommand{\cO}{\mathcal{O}}

\begin{document}
\title{Twisted B-splines in the complex plane}

\author{S.R. Das}
\address{Department of Mathematics, Indian Institute of Technology, Madras}
%\curraddr{Department of Mathematics and Statistics,
%Case Western Reserve University, Cleveland, Ohio 43403}
\email{santiranjandas100@gmail.com}
\author{P. Massopust}
\address{Center of Mathematics, Technical University of Munich, Garching by Munich}
\email{massopust@ma.tum.de}

\author{R. Radha}
\address{Department of Mathematics, Indian Institute of Technology, Madras}
\email{radharam@iitm.ac.in}
%    General info
\subjclass[2010]{Primary 42C15; Secondary 41A15, 43A30}

%\date{January 1, 2001 and, in revised form, June 22, 2001.}


\keywords{$B$-splines; Gramian; {\it Hilbert-Schmidt} operator; {\it Riesz} sequence; Twisted convolution; {\it Weyl} transform; Nonstationary multiresolution analysis.}
\begin{abstract}
In this paper, we introduce the new class of twisted $B$-splines and study some properties of these B-splines. We also investigate the system of twisted translates and the wavelets corresponding to these twisted $B$-splines.
\end{abstract}
\maketitle
\section{Introduction}
Let $X := \{x_0 < x_1 < \cdots < x_k<x_{k+1}\}$ be a set of knots on the real line $\R$. A spline function, or for short a spline, of order $n$ on $[x_0,x_{k+1}]$ with knot set $X$ is a function $f: [x_0,x_{k+1}]\to\mathbb{R}$ such that
\begin{enumerate}
\item[(i)] on each interval $(x_{i-1},x_i)$, $i = 0,1,\ldots, k+1$, $f$ is a polynomial of order at most $n$;
\item[(ii)]   $f\in C^{n-2}[x_0,x_{k+1}]$. In the case $n = 1$, we define $C^{-1}[x_0,x_{k+1}]$ to be the space of piecewise constant functions on $[x_0,x_{k+1}]$.
\end{enumerate}
The function $f$ is called a {cardinal spline} if the knot set $X$ is a contiguous subset of $\Z$. As a spline of order $n$ contains $n(k+1)$ free parameters on $[x_0,x_{k+1}]$ and has to satisfy $n-1$ differentiability conditions at the $k$ interior knots $x_1, \ldots, x_k$, the set ${S}_{X,n}$ of all spline functions $f$ of order $n$ over the knot set $X$ forms a real vector space of dimension $n+k$. As the space of cardinal splines of order $n$ over a finite knot set $X$ is finite-dimensional, a convenient and powerful basis for ${S}_{X,n}$ is given by the family of polynomial cardinal B-splines introduced by Curry \& Schoenberg \cite{CS}. For a detailed study of splines and wavelets we refer to \cite{Ch} and for splines and fractal functions we refer to \cite{peter}.\\

The theory of $B$-splines can be extended to $\c=\R^2$ by using the tensor product of functions, namely, $\mathbb{B}_n(x,y) := B_n(x)B_n(y),\ x,y\in\R$. The $B$-splines considered in this paper are totally different. In order to define these new  \emph{twisted $B$-splines}, we make use of the concept of  twisted convolution. This is a non-standard convolution which arises while transferring the convolution from the {\it Heisenberg} group to the complex plane.\\

The {\it Heisenberg} group $\mathbb{H}$ is a nilpotent {\it Lie} group whose underlying manifold is $\R\times\R\times\R$ endowed with a group operation defined by 
\[
(x,y,t)(x^\prime,y^\prime,t^\prime):=(x+x^\prime,y+y^\prime,t+t^\prime+\tfrac{1}{2}(x^\prime y-y^\prime x)),
\]
and where the {\it Haar} measure is {\it Lebesgue} measure $dx\,dy\,dt$ on $\R^3$. By the {\it Stone--von Neumann} theorem, every infinite dimensional irreducible unitary representation on $\mathbb{H}$ is unitarily equivalent to the representation $\pi_\lambda$ given by
\begin{align*}
\pi_\lambda(x,y,t)\p(\xi)=e^{2\pi i\lambda t}e^{2\pi i\lambda(x\xi+\frac{1}{2}xy)}\p(\xi+y),
\end{align*}
for $\p\in L^2(\R)$ and $\lambda\in \R^\times:=\R\setminus\{0\}$. This representation $\pi_\lambda$ is called the {\it Schr\"{o}dinger} representation of the {\it Heisenberg} group. For $f,g\in L^1(\mathbb{H})$, the group convolution of $f$ and $g$ is defined by
\begin{align*}
f*g(x,y,t):=\int_{\mathbb{H}}f\big((x,y,t)(u,v,s)^{-1}\big)g(u,v,s)\ du\,dv\,ds.
\end{align*}
This can be reduced to $\R^2$ as a twisted convolution in the following way. For $f\in L^1(\mathbb{H})$, define
\begin{align*}
f^\lambda(x,y) :=\int_\R e^{2\pi i\lambda t}f(x,y,t)\ dt.
\end{align*}
Now for $F,G\in L^1(\R^2)$, define
\begin{align*}
F*_\lambda G(x,y) :=\iint_{\R^2}F(x-u,y-v)G(u,v)\ e^{\pi i\lambda(uy-vx)}\ du\,dv.
\end{align*}
The above operations are called $\lambda$-twisted convolutions on $\R^2$. One can show that for $f,g\in L^1(\mathbb{H})$, $(f*g)^\lambda=f^\lambda *_\lambda g^\lambda$. In particular, when $\lambda:=1$, the $\lambda$-twisted convolution is simply called the twisted convolution and denoted by $F\times G$ for $F,G\in L^1(\R^2)$.\\

Now let $\pi(x,y):=\pi_1(x,y,0)$. Then for $f\in L^1(\R^2)$, the {\it Weyl} transform of $f$ is defined to be
\begin{align*}
W(f):=\int_{\R^2}f(x,y)\pi(x,y)\ dx\,dy,
\end{align*}
with $\pi(x,y)\p(\xi)=e^{2\pi i(x\xi+\frac{1}{2}xy)}\p(\xi+y)$, $\p\in L^2(\R)$. In order to study problems concerned with the group {\it Fourier} transform on $\mathbb{H}$, an important technique is to take the partial {\it Fourier} transform in the $t$-variable and reduce the study to the case $\R^2$. Then the analysis on the {\it Heisenberg} group can be studied by first looking into an analysis based on the {\it Weyl} transform and twisted convolution. As we intend to study in the future $B$-splines on the {\it Heisenberg} group, we lay the ground work for it in this paper by studying $B$-splines on $\c$ using twisted convolution. It is important to mention here that $L^1(\R^2)$ turns out to be a non-commutative {\it Banach} algebra under twisted convolution, unlike ordinary convolution on $\R^2$. For a study of the {\it Heisenberg} group and related topics, we refer to \cite{BCT,follandphase,thangavelu}.\\

In \cite{saswatajmaa}, Radha and Adhikari studied the problem of characterizing the system of twisted translates to be a  frame sequence or a {\it Riesz} sequence. This study was later extended to the {\it Heisenberg} group in \cite{saswatahouston}. In \cite{saswatacollec}, Radha and Adhikari studied some properties of twisted translates of a square-integrable function on $\c$ and later extended those results to the {\it Heisenberg} group. Recently, in \cite{aratijmpa}, the orthonormality of a wavelet system associated with twisted translates and left translates on the {\it Heisenberg} group were investigated.\\

It is worth mentioning here that the presence of the ``twist" term in the definition of a twisted $B$-spline complicates the computations and induces an entirely different geometric structure from that of the classical $B$-splines on the real line. \\

The new class of twisted B-splines constructed in this paper have the same support as the tensor products of the classical B-splines but possess higher regularity; instead of being piecewise polynomial they are piecewise analytic. This latter property makes the twisted B-splines suitable for approximating or interpolating higher dimensional surfaces which have locally higher than polynomial regularity. As they generate nonstationary multiresolution analyses of $L^2(\R^2)$ they may also be employed for Galekin-type schemes.  \\

We organize the paper as follows. In Section $2$, we provide the necessary background for the remainder of this article. In Section $3$, we define twisted $B$-splines and study some of their properties. In Section $4$, we investigate the system of twisted translates of the twisted $B$-splines and prove that this collection is a {\it Riesz} basis for the second order $B$-spline. In Section $5$, we introduce the notion of {\it multiresolution analysis} on $L^2(\R^2)$ using twisted translations and standard dilations. However, we find that even in the simplest case of the first order twisted $B$-spline $\p_1$ this notion does not help to obtain a wavelet basis for $L^2(\R^2)$.  We need to employ the concept of \textit{nonstationary  multiresolution analysis} and use ``certain" twisted translations and dilations. This is illustrated in Example \ref{p}.\\

\section{Notation and Background}

Let $0\neq\h$ be a separable Hilbert space.
\begin{definition}
A sequence $\{f_k:~k\in\N\}$ of elements in $\h$ is said to be a {\it Bessel} sequence for $\h$ if there exists a constant $B>0$ satisfying
\begin{align*}
\sum_{k\in\N}\mid\langle f,f_k\rangle\mid^2\hspace{1 mm}\leq B\|f\|^2,\hspace{.75 cm}\forall\ f\in\h.
\end{align*}
\end{definition}
\begin{definition}
A sequence of the form $\{Ue_k:~k\in\N\}$, where $\{e_k:~k\in\N\}$ is an orthonormal basis of $\h$ and $U$ is a bounded invertible operator on $\h$, is called a Riesz basis. If $\{f_k:~k\in\N\}$ is a Riesz basis for $\overline{\Span\{f_k:~k\in\N\}}$, then it is called a Riesz sequence. 
\end{definition}
Equivalently, $\{f_k:~k\in\N\}$ is said to be a Riesz sequence if there exist constants $A,B>0$ such that
\begin{align*}
A\|\{c_{k}\}\|_{\ell^2(\N)}^2\leq\bigg\|\sum_{k\in\N}c_kf_k\bigg\|^2\leq B\|\{c_{k}\}\|_{\ell^2(\N)}^2,
\end{align*}
for all finite sequences $\{c_{k}\}\subset\ell^2(\N)$.
\begin{definition}
The {\it Gramian} $G$ associated with a {\it Bessel} sequence $\{f_k:~k\in\N\}$ is a bounded operator on $\ell^2(\N)$ defined by
\begin{align*}
G\{c_{k}\}:=\bigg\{\sum_{k\in\N}\big\langle f_k,f_j\big\rangle c_k\bigg\}_{j\in\N}.
\end{align*}
\end{definition}
It is well known that $\{f_k:~k\in\N\}$ is a {\it Riesz} sequence iff
\begin{align*}
A\|\{c_{k}\}\|_{\ell^2(\N)}^2\leq\big\langle G\{c_{k}\},\{c_{k}\}\big\rangle\leq B\|\{c_{k}\}\|_{\ell^2(\N)}^2
\end{align*}
for constants $A, B > 0$.
\begin{definition}
A closed subspace $V\subset L^2(\R)$ is called a shift-invariant space if $f\in V\Rightarrow T_kf\in V$ for any $k\in\z$, where $T_k$ denotes the translation operator $T_kf(y):=f(y-k)$. In particular, if $\p\in L^2(\R)$ then $V(\p)=\overline{\Span\ \{T_k\p:k\in\z\}}$ is called a principal shift-invariant space.
\end{definition}
For a detailed study of {\it Riesz} bases on $\mathcal{H}$ and shift-invariant spaces on $L^2(\R)$, we refer to \cite{C}.
\begin{definition}
Let $\psi\in L^2(\R)$ and $j,k\in\z$. Define $\psi_{j,k}(x):=2^{j/2}\psi(2^jx-k),\ x\in\R$. The latter can also be written as
\begin{align*}
\psi_{j,k}=D_{2^j}T_k\psi\ \ \ j,k\in\z,
\end{align*}
where $D_{2^j}f(x):=2^{j/2}f\big(2^{j}x\big)$ and $T_kf(x):=f(x-k)$, $f\in L^2(\R)$, are unitary operators on $L^2(\R)$. If $\{\psi_{j,k}:j,k\in\z\}$ is an orthonormal basis for $L^2(\R)$ then the function $\psi$ is called a wavelet or mother wavelet and $\{\psi_{j,k}:j,k\in\z\}$ is called a wavelet basis for $L^2(\R)$.   
\end{definition}
The following definition is useful in constructing wavelets for $L^2(\R)$. This notion was first introduced by {\it Meyer} in \cite{meyer} and independently studied and developed by {\it Mallat} in \cite{mallat}.
\begin{definition}
A sequence $\{V_j\}_{j\in\z}$ of closed subspaces of $L^2(\R)$ is said to form a {\it multiresolution analysis} for $L^2(\R)$ if the following conditions hold:
\begin{enumerate}[(i)]
\item $V_j\subset V_{j+1},\ \forall\ j\in\z$.
\item $\overline{\bigcup\limits_{j\in\z}V_j}=L^2(\R)$ and $\bigcap\limits_{j\in\z} V_j=\{0\}$.
\item $f\in V_j\;\Longleftrightarrow\; f(2\,\cdot\,)\in V_{j+1},\ \forall\ j\in\z$.
\item There exists $\p\in V_0$ such that $\{T_k\p:k\in\z\}$ is an orthonormal basis for $V_0$.
\end{enumerate}
The function $\p$ is called a scaling function of the given {\it multiresolution analysis}.  
\end{definition}
\begin{example}
Let $\p:=\chi_{[0,1]}$, where $\chi_{[0,1]}$ is the characteristic function of $[0,1]$. Let $V_0=\overline{\Span\ \{T_k\p:k\in\z\}}$. Define
\begin{align*}
\psi(x)&:=\begin{cases}
1,& 0\leq x<\frac{1}{2};\\
-1,& \frac{1}{2}\leq x<1;\\
0,& \text{otherwise}.
\end{cases}
\end{align*}
Then $\psi$ is called the {\it Haar} wavelet.
\end{example}
For a study of wavelets we refer to, e.g., \cite{daubechies}.
{\begin{definition}
A \textit{nonstationary multiresolution analysis} of $L^2(\R)$ is a sequence of closed subspaces $\{V_j:j\in\z\}$ of $L^2(\R)$ such that
\begin{enumerate}[(i)]
\item $V_j\subset V_{j+1},\ \forall\ j\in\z$.
\item $\overline{\bigcup\limits_{j\in\z}V_j}=L^2(\R)$ and $\bigcap\limits_{j\in\z} V_j=\{0\}$.
\item For $j\in\z$, there exists $\Phi_j\in V_j$ such that $\{2^{j/2}\Phi_j(2^jx-k):k\in\z\}$ is a {\it Riesz} basis for $V_j$.
\end{enumerate} 
\end{definition}
Then as in the case of usual {\it multiresolution analysis}, by taking $V_0\oplus W_0=V_1$ and $V_j\oplus W_j=V_{j+1}$, one arrives at a sequence $\{\Psi_j:j\in\z\}$ such that $\{2^{j/2}\Psi_j(2^jx-k):k\in\z\}$ is a {\it Riesz} basis for $W_j$, which in turn leads to the \textit{Riesz} basis for $L^2(\R)$. We refer to \cite{boor, bastin, blu} in this connection.}
\begin{definition}
Let $\chi_{[0,1]}$ denote the characteristic function of $[0,1]$. For $n\in \N$, set
\begin{align*}
&B_{1}:[0,1]\to [0,1], \quad x\mapsto \chi_{[0,1]}(x);\\
&B_{n} := B_{n-1}\ast B_{1}, \quad 1 < n\in \N.\numberthis \label{15}
\end{align*}
An element of the discrete family $\{B_n\}_{n\in \N}$ is called a \emph{(cardinal) polynomial B-spline of order $n$}.
\end{definition}

\begin{remark}
The adjective ``cardinal'' refers to the fact that $B_n$ is defined over the set of knots $\{0,1, \ldots, n\}$.
\end{remark}

Using this definition, it can be shown that $B_n$ has a closed representation of the form
\begin{align*}
B_n (x)  = \frac{1}{\Gamma (n)}\,\sum_{k=0}^{n} (-1)^k \binom{n}{k} (x-k)_+^{n-1},
\end{align*}
where $\Gamma$ denotes the Euler Gamma function and $x_+^p := \max\{0,x\}^p$ a truncated power function.

Now \eqref{15} implies that the Fourier representation of $B_n$ is given by
\[
\widehat{B}_{n}(\omega) := \mathcal{F}(B_n)(\omega)  := \int_\R B_n(x) e^{- i\omega x} dx = \left(\frac{1-e^{- i\omega}}{i\omega} \right)^{n}.
\]
Next, we summarize some of the properties of the cardinal B-splines $B_n$ that are relevant for the remainder of this paper.
\begin{enumerate}[(i)]
\item $\supp B_n = [0,n]$. 
\item $B_n > 0$ on $(0,n)$.
\item Partition of Unity: $\sum\limits_{k\in \Z} B_n(x-k) = 1$, for all $x\in \R$.
\item The collection $\{B_n : n\in \N\}$ constitutes a family of piecewise polynomial functions with $B_n\in C^{n-2} [0,n]$, $n\in \N$. In the case $n=1$, we define $C^{-1}[0,1]$ as the family of piecewise constant functions on $[0,1]$.
\end{enumerate}
A particular property of B-splines which we consider in the setting of the present paper is the following.

\begin{theorem}\cite[Theorem 9.2.6]{C}
For each $n\in \N$, the sequence $\{T_kB_n\}_{k\in \z}$ is a Riesz sequence.
\end{theorem}

In addition, polynomial (cardinal) B-splines satisfy a two-scale refinement equation of the form
\[
B_n (x) = \sum_{k=0}^n \frac{1}{2^{m-1}}\binom{n}{k} B_n(2x-k),
\]
and can be used to define certain classes of wavelets. We observe that when $B_1=\x_{[0,1]}$, the corresponding wavelet is the {\it Haar} wavelet. For more details, we refer to, for instance, \cite{Ch}.\\

The {\it Weyl} transform, defined in the Introduction, maps $L^1(\R^2)$ into the space of bounded operators on $L^2(\R)$, denoted by $\mathcal{B}(L^2(\R))$. Moreover, the {\it Weyl} transform $W(f)$ is an integral operator whose kernel $K_f(\xi,\eta)$ is given by
\begin{align*}
K_f(\xi,\eta)=\int_\R f(x,\eta-\xi)e^{\pi ix(\xi+\eta)}\ dx.
\end{align*}
The map $W$ can be uniquely extended to a bijection from the class of tempered distributions $S^\prime(\R^2)$ onto the space of continuous linear maps from $S(\R)$ into $S^\prime(\R)$. Here, $S(\R)$ denotes the Schwarz space of real-valued functions on $\R$.\\

When $f\in L^2(\R^2)$, $W(f)$ turns out to be a {Hilbert-Schmidt} operator on $L^2(\R)$. Furthermore, $W(f)$ enjoys the properties reminiscent of those of the {\it Fourier} transform on $\R$. For instance, we have the {\it Plancherel} formula 
\[
\|W(f)\|_\b=\|f\|_{L^2(\R^2)},
\]
where $\b:=\b(L^2(\R))$ is the space of {Hilbert-Schmidt} operators on $L^2(\R)$. Moreover,
\begin{align*}
\big\langle W(f),W(g)\big\rangle_\b=\big\langle f,g\big\rangle_{L^2(\R^2)}=\big\langle K_f,K_g\big\rangle_{L^2(\R^2)}.
\end{align*}
The inversion formula is given by
\begin{align*}
f(x,y)=\tr(\pi(x,y)^*W(f)),
\end{align*}
where $\tr$ denotes the trace, and, in addition, 
\[
W(f\times g)=W(f)W(g), \quad\text{for $f,g\in L^2(\R^2)$.}
\]    
\begin{definition}
Let $\phi\in L^2(\R^2)$. For $(k,l)\in\z^2$, the twisted translation of $\phi$, denoted by $\t\phi$, is defined to be
\begin{align}\label{11}
\t\phi(x,y):=e^{\pi i(lx-ky)}\phi(x-k,y-l),\ \ \ (x,y)\in\R^2.
\end{align}
\end{definition}
\begin{definition}
The twisted shift-invariant space of $\phi$, denoted by $V^t(\phi)$, and defined by
\begin{align*}
V^t(\phi):=\overline{\Span\{\t\phi:k,l\in\z\}}
\end{align*}
is a closed subspace of $L^2(\R^2)$.
\end{definition}
In order to give a reader a feeling for twisted translation, we provide some properties which were proved in \cite{saswatajmaa}.
\begin{enumerate}[(i)]
\item The adjoint $(\t)^*$ of $\t$ is $T_{(-k,-l)}^t$.
\item $T_{(k_1,l_1)}^tT_{(k_2,l_2)}^t=e^{-\pi i(k_1l_2-l_1k_2)}T_{(k_1+k_2,l_1+l_2)}^t$.
\item $\t$ is an unitary operator on $L^2(\R^2)$, for each $(k,l)\in\z^2$.
\item  The {\it Weyl} transform of $\t\p$ is given by $W(\t\p)=\pi(k,l)W(\p)$.
\end{enumerate}
%For $f\in L^1(\R^2)$, the {\it Weyl} transform of $f$ is defined by
%\begin{align*}
%W(f)=\int_{\R^2}f(x,y)\pi(x,y)\ dxdy,
%\end{align*}
%where $\pi(x,y)=\pi_1(x,y,0)$ with
%\begin{align*}
%\pi_\lambda(x,y,t)\phi(\xi)=e^{2\pi i\lambda t}e^{2\pi i\lambda(x\cdot\xi+\frac{1}{2}x\cdot y)}\phi(\xi+y),\ \ \ \phi\in L^2(\R).
%\end{align*}
%The kernel of $W(f)$ is
%\begin{align*}
%K_f(\xi,\eta)=\int_\R f(x,\eta-\xi)e^{\pi ix(\xi+\eta)}\ dx.
%\end{align*}
%For $f,g\in L^1(\R^2)$, the twisted convolution of $f$ and $g$ is defined to be
%\begin{align*}
%f\times g(x,y)=\iint_{\R^2}f(x-u,y-v)g(u,v)e^{\pi i(uy-vx)}\ dudv.
%\end{align*}
\section{Definition and Various Properties of Twisted B-splines}
\begin{definition}
Let $\phi_1(x,y):=\x_{[0,1)}(x)\x_{[0,1)}(y)$. Define
\begin{align*}
\phi_n= \underset{i = 1}{\overset{n-1}{\times}} \phi_i := \phi_1\times\phi_1\times\cdots\times\phi_1,\quad(\text{$(n-1)$--fold twisted convolution}).
\end{align*}
Then $\p_n$ is called an $n$-th order twisted B-spline. 
\end{definition}
Using this definition, we can derive an explicit formula for $\phi_2(x,y)$.
\begin{align*}
\phi_2(x,y)&=\p_1\times\p_1(x,y)\\
&=\iint_{\R^2}\p_1(x-u,y-v)\p_1(u,v)e^{\pi i(uy-vx)}\ du\,dv\\
&=\iint_{[0,1)\times [0,1)}\x_{[0,1)}(x-u)\x_{[0,1)}(y-v)e^{\pi i(uy-vx)}\ du\,dv\\
&=\bigg(\int_0^1\x_{[0,1)}(x-u)e^{\pi iuy}\ du\bigg)\bigg(\int_0^1\x_{[0,1)}(y-v)e^{-\pi ivx}\ dv\bigg)\\
&=\bigg(\int_{u\in[0,1)\cap(x-1,x]}e^{\pi iuy}\ du\bigg)\bigg(\int_{v\in[0,1)\cap(y-1,y]}e^{-\pi ivx}\ dv\bigg).
\end{align*}
Then a straight forward computation leads to the following:
\begin{align*}
\p_2(x,y)&=\begin{cases}
\frac{1}{\pi^2xy}(e^{\pi ixy}-1)(e^{-\pi ixy}-1),& (x,y)\in(0,1]\times(0,1];\\
\frac{1}{\pi^2xy}(e^{\pi ixy}-1)(e^{-\pi ix}-e^{-\pi ix(y-1)}), & (x,y)\in(0,1]\times(1,2];\\
\frac{1}{\pi^2xy}(e^{\pi iy}-e^{\pi i(x-1)y})(e^{-\pi ixy}-1), & (x,y)\in(1,2]\times(0,1];\\
\frac{1}{\pi^2xy}(e^{\pi iy}-e^{\pi i(x-1)y})(e^{-\pi ix}-e^{-\pi ix(y-1)}), & (x,y)\in(1,2)\times(1,2);\\
0, & \text{otherwise}.
\end{cases}
\end{align*}
Using trigonometric identities, $\p_2$ can be written in the form
\begin{align}\label{1}
\p_2(x,y)&=\begin{cases}
\frac{2}{\pi^2xy}\big(1-\cos(\pi xy)\big), & (x,y)\in(0,1]\times(0,1];\\
\frac{2}{\pi^2xy}\big(\cos(\pi x(y-1))-\cos(\pi x)\big), & (x,y)\in(0,1]\times(1,2];\\
\frac{2}{\pi^2xy}\big(\cos(\pi (x-1)y)-\cos(\pi y)\big), & (x,y)\in(1,2]\times(0,1];\\
\frac{2}{\pi^2xy}\big(\cos(\pi (x-y))-\cos(\pi (x+y-xy))\big), & (x,y)\in(1,2)\times(1,2);\\
0, & \text{otherwise}.
\end{cases}
\end{align}
Figure \ref{fig1} displays the graphs of the classical second order tensor product B-spline $\mathbb{B}_2 = B_2\otimes B_2$ and the second order twisted B-spline $\p_2$.

\begin{figure}[h!]
\begin{center}
\includegraphics[width=7cm, height= 5cm]{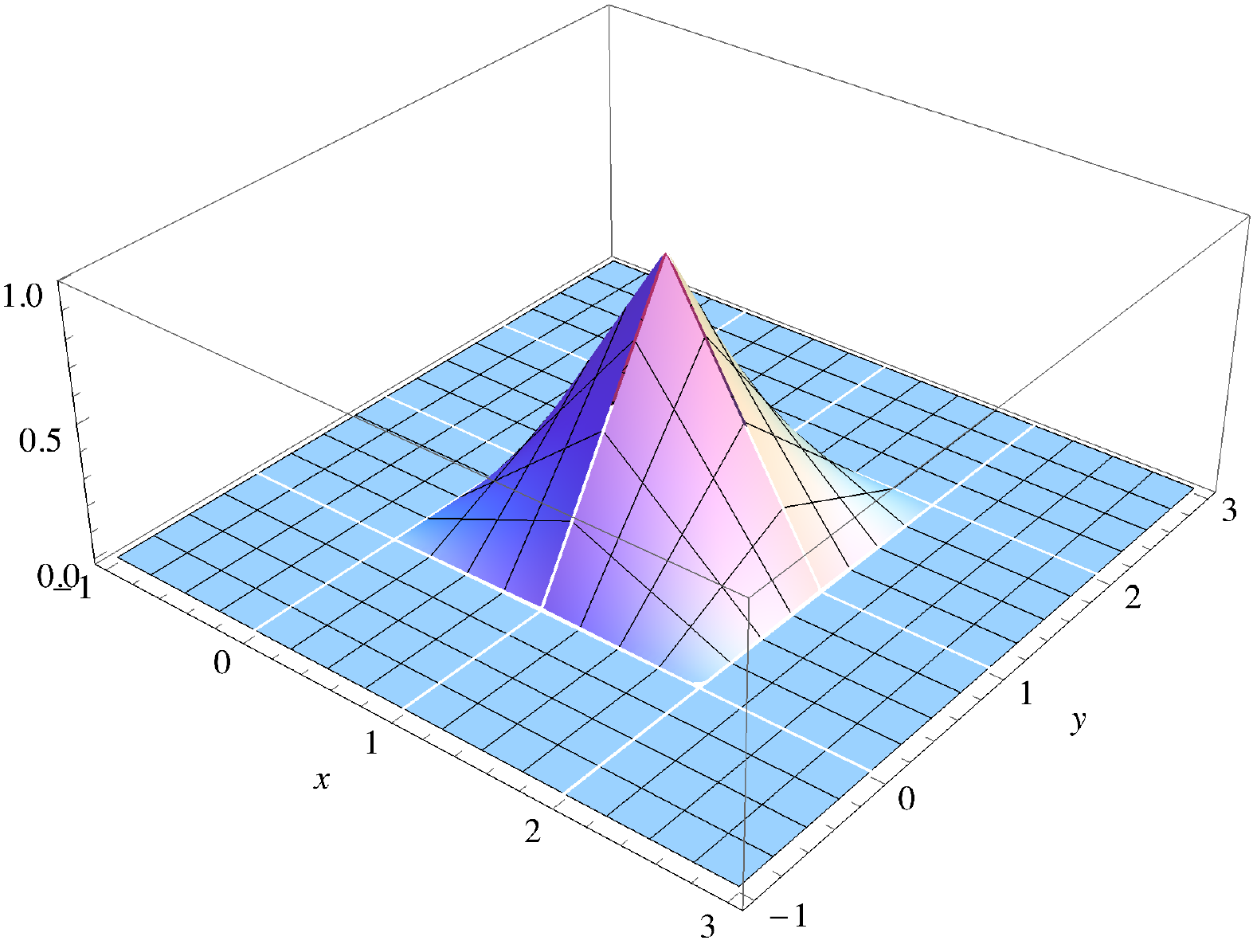}\qquad\includegraphics[width=7cm, height= 5cm]{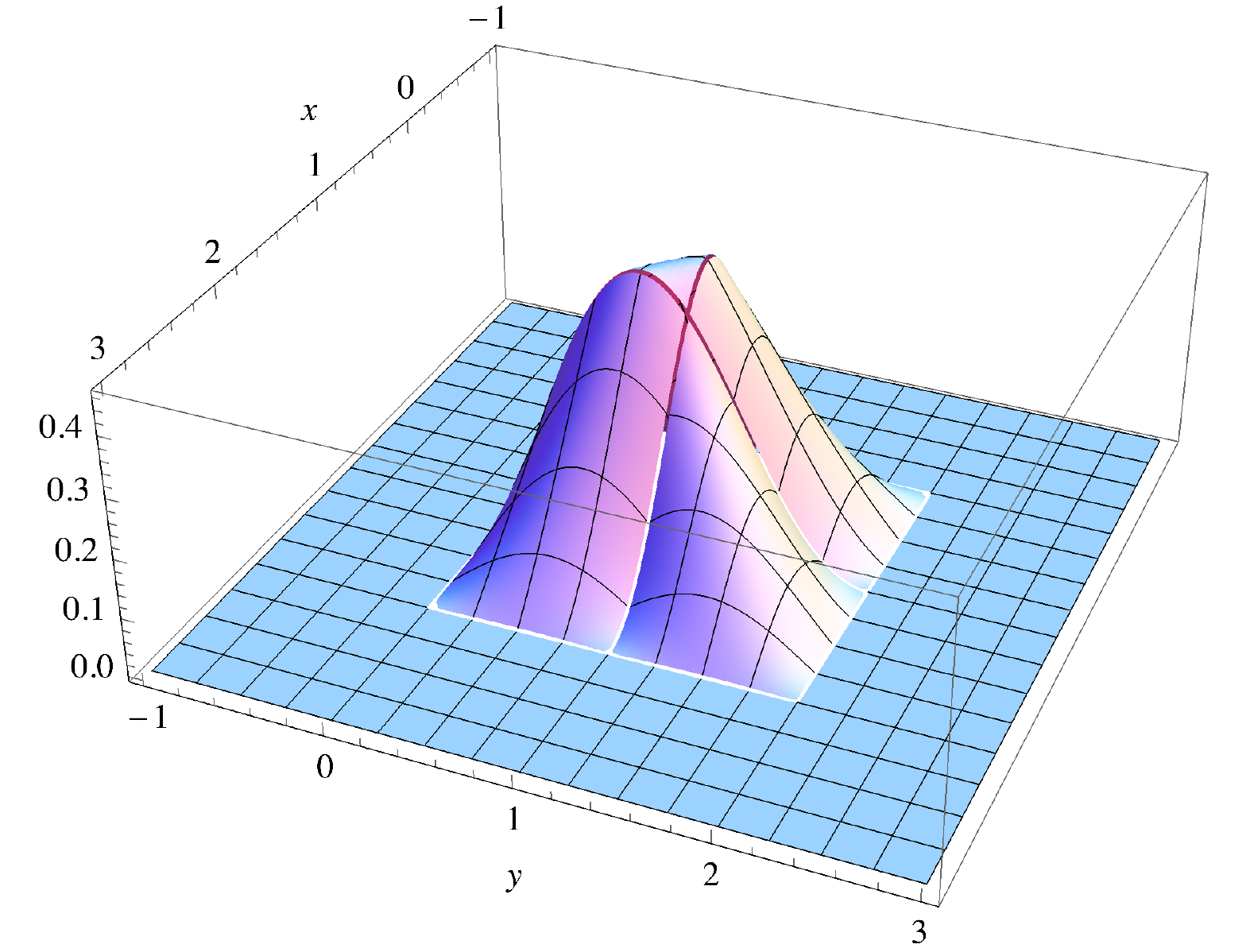}
\caption{The classical B-spline $\mathbb{B}_2$ (left) and twisted B-spline $\phi_2$  (right).}\label{fig1}
\end{center}
\end{figure}

We note that
\begin{align*}
\p_{n+1}(x,y)&=\p_n\times\p_1(x,y)\\
&=\iint_{\R^2}\p_n(x-u,y-v)\x_{[0,1)}(u)\x_{[0,1)}(v)e^{\pi i(uy-vx)}\ du\,dv.
\end{align*}
Thus
\begin{align}\label{2}
\p_{n+1}(x,y)=\int_{0}^1\int_{0}^1\p_n(x-u,y-v)e^{\pi i(uy-vx)}\ dv\,du.
\end{align}
\begin{proposition}
$\supp\phi_n=[0,n]\times[0,n],\ \forall\ n\in\N$.
\end{proposition}
\begin{proof}
By \eqref{1}, we observe that $\supp\p_2=[0,2]\times[0,2]$. Assume that $\supp\phi_{n}=[0,n]\times[0,n]$. Now,
\begin{align*}
\phi_{n+1}(x,y)=\int_{u\in[x-1,x]\cap[0,n]}\int_{v\in[y-1,y]\cap[0,n]}\phi_n(u,v)\ e^{\pi i(vx-uy)}\ du\,dv,
\end{align*}
which means $0< x< n+1,\ 0<y<n+1$. Thus if $(x,y)\notin(0,n+1)\times(0,n+1)$, then it follows that $\phi_{n+1}(x,y)=0$ showing that $\supp \phi_{n+1}=[0,n+1]\times[0,n+1]$.
\end{proof}
\begin{proposition}
We have
\begin{align*}
\iint_{\R^2}\phi_1(x,y)\ dxdy=1
\end{align*}
and
\begin{align*}
\iint_{\R^2}\phi_2(x,y)\ dxdy=\frac{1}{\pi^2}\big(-\gamma-\log\pi+\Ci(\pi)+i\Si(\pi)\big)\big(-\gamma-\log\pi+\Ci(\pi)-i\Si(\pi)\big),
\end{align*}
where $\gamma$ denotes the Euler--Mascheroni constant, and 
\[
\Ci (z) := \int_0^z \frac{\sin t}{t}\,dt,\quad \Si (z) :=  \int_0^z \frac{\cos t}{t}\,dt, \quad |\arg z| < \pi,
\] 
the {\it cosine} and {\it sine} integrals, respectively \cite{leb}. 

For $n\geq 1$,
\begin{align*}
\iint_{\R^2}\phi_{n+1}(x,y)\ dxdy=\iint_QL_n(u,v)\ du\,dv,
\end{align*}
where $Q=[0,1]\times[0,1]$ and $L_n$ is evaluated recursively as follows:
\begin{align*}
L_{n+1}(u,v)=\iint_QL_n(p+u,q+v)e^{\pi i(uq-vp)}\ dp\,dq,
\end{align*}
with
\begin{align*}
L_1(u,v)=\frac{1}{\pi^2uv}(e^{\pi iu}-1)(e^{-\pi iv}-1).
\end{align*}
\end{proposition}
\begin{proof}
For $n=1$, we obtain
\begin{align*}
\iint_{\R^2}\phi_1(x,y)\ dxdy&=\iint_{\R^2}\x_{[0,1)}(x)\x_{[0,1)}(y)\ dx\,dy =1.
\end{align*}
For $n=2$, making use of \eqref{2} and Fubini's theorem, we get
\begin{align*}
\iint_{\R^2}\phi_2(x,y)\ dxdy&=\iint_{\R^2}\bigg(\int_{0}^1\int_{0}^1\p_1(x-u,y-v)e^{\pi i(uy-vx)}\ dv\,du\bigg)dx\,dy\\
&=\iint_{Q}\bigg(\iint_{\R^2}\p_1(x-u,y-v)e^{\pi i(uy-vx)}\ dx\,dy\bigg)du\,dv\\
&=\iint_Q\bigg(\iint_Qe^{\pi i(us-vr)}\ dr\,ds\bigg)du\,dv,
\end{align*}
where we applied the change of variables $x-u=r$ and $y-v=s$. However,
\begin{align*}
\iint_Qe^{\pi i(us-vr)}\ dr\,ds&=\frac{1}{\pi^2uv}(e^{\pi iu}-1)(e^{-\pi iv}-1).
\end{align*}
Thus
\begin{align*}
\iint_{\R^2}\phi_2(x,y)\ dx\,dy&=\frac{1}{\pi^2}\bigg(\int_0^1\frac{e^{\pi iu}-1}{u}\ du\bigg)\bigg(\int_0^1\frac{e^{-\pi iv}-1}{v}\ dv\bigg)\\
&=\frac{1}{\pi^2}\big(-\gamma-\log\pi+\Ci(\pi)+i\Si(\pi)\big)\big(-\gamma-\log\pi+\Ci(\pi)-i\Si(\pi)\big).
\end{align*}
Now, again by using \eqref{2},
\begin{align*}
\iint_{\R^2}\phi_{n+1}(x,y)\ dx\,dy&=\iint_{\R^2}\bigg(\int_{0}^1\int_{0}^1\p_n(x-u,y-v)e^{\pi i(uy-vx)}\ dv\,du\bigg)dx\,dy\\
&=\iint_{Q}\bigg(\iint_{\R^2}\p_n(x-u,y-v)e^{\pi i(uy-vx)}\ dx\,dy\bigg)du\,dv\\
&=\iint_{Q}\bigg(\iint_{\R^2}\p_n(r,s)e^{\pi i(us-vr)}\ dr\,ds\bigg)du\,dv.\numberthis \label{4}
\end{align*}
Thus, we need to compute
\begin{align*}
\iint_{\R^2}\p_n(r,s)e^{\pi i(us-vr)}\ dr\,ds,\ \ \ \forall\ n\in\N.
\end{align*}
We shall use induction on $n$. First, for $n=1$, we set
\begin{align*}
\iint_{\R^2}\p_1(r,s)e^{\pi i(us-vr)}\ dr\,ds&=\bigg(\int_0^1e^{-\pi ivr}\ dr\bigg)\bigg(\int_0^1e^{\pi ius}\ ds\bigg)\\
&=\frac{1}{\pi^2}\bigg(\frac{e^{\pi iu}-1}{u}\bigg)\bigg(\frac{e^{-\pi iv}-1}{v}\bigg)\\
&=: L_1(u,v),\numberthis \label{3}.
\end{align*}
Now, assume that the result is true for $n\geq 1$, that is,
\begin{align}\label{5}
\iint_{\R^2}\p_n(r,s)e^{\pi i(us-vr)}\ dr\,ds=L_n(u,v).
\end{align}
Using \eqref{2}, Fubini's theorem, and then applying a change of variables, yields
\begin{align*}
\iint_{\R^2}\p_{n+1}(r,s)e^{\pi i(us-vr)}\ dr\,ds&=\iint_Q\bigg(\iint_{\R^2}\p_n(r,s)e^{\pi i((p+u)s-(q+v)r)}\ dr\,ds\bigg)e^{\pi i(uq-vp)}\ dp\,dq\\
&=\iint_QL_n(p+u,q+v)e^{\pi i(uq-vp)}\ dp\,dq.
\end{align*}
Hence, if we denote the left-hand side of the above equation by $L_{n+1}(u,v)$, we obtain
\begin{align*}
L_{n+1}(u,v)=\iint_QL_n(p+u,q+v)e^{\pi i(uq-vp)}\ dpdq.
\end{align*}
Therefore, by \eqref{3}, we get
\begin{align*}
L_2(u,v)&=\iint_QL_1(p+u,q+v)e^{\pi i(uq-vp)}\ dpdq\\
&=\frac{1}{\pi^2}\bigg(\int_0^1\frac{e^{\pi i(p+u)}-1}{p+u}e^{-\pi ivp}\ dp\bigg)\bigg(\int_0^1\frac{e^{-\pi i(q+v)}-1}{q+v}e^{\pi iuq}\ dq\bigg)\\
&=\frac{1}{\pi^2}\big(\Ei(-\pi iuv)-\Ei(-\pi iu(v-1))-\Ei(-\pi i(u+1)v)+\Ei(-\pi i(u+1)(v-1))\big)\times\\
&\hspace{2 cm}\big(\Ei(\pi iuv)-\Ei(\pi iu(v+1))-\Ei(\pi i(u+1)v)+\Ei(\pi i(u+1)(v+1))\big),
\end{align*}
with $\Ei$ denoting the exponential integral
\begin{align*}
\Ei(z)=-\int_{-z}^\infty \frac{e^{-t}}{t}\ dt,\quad \abs{\arg z} < \pi,
\end{align*}
where the principal value of the integral is taken and the branch cut ranges from $-\infty$ to $0$. (Cf. \cite{leb}.) Now, using \eqref{5} in \eqref{4} and induction on $n$, we can obtain the values for
\begin{align*}
\iint_{\R^2}\phi_{n+1}(x,y)\ dx\,dy
\end{align*}
for all $2\leq n \in \N$.
\end{proof}
\begin{proposition}
The kernel of the {\it Weyl} transform of $\p_1$ is given by
\begin{align*}
K_{\p_1}(\xi,\eta)&= 
e^{\frac{\pi i}{2}(\xi+\eta)}\sinc\left(\frac{\xi+\eta}{2}\right)\x_{[0,1]}(\eta-\xi),\ \ \ \xi,\eta\in\R.
\end{align*}
For $n\geq 2$, the kernel of the {\it Weyl} transform of $\p_n$ can be evaluated recursively as follows:
\begin{align*}
K_{\p_n}(\xi,\eta)=e^{\pi i\eta}\int_0^1e^{-\frac{\pi i}{2}y_{n-1}}\sinc\left(\frac{2\eta-y_{n-1}}{2}\right)K_{\p_{n-1}}(\xi,\eta-y_{n-1})\ dy_{n-1}.
\end{align*}
\end{proposition}
\begin{proof}
Recall that $\phi_1(x,y)=\x_{[0,1)}(x)\x_{[0,1)}(y)$ and that the {\it Weyl} transform of $\p_1$ is an integral operator whose kernel $K_{\p_1}$ is given by
\begin{align*}
K_{\p_1}(\xi,\eta)&=\int_{\R}\p_1(x,\eta-\xi)e^{\pi ix(\xi+\eta)}\ dx\\
&=\x_{[0,1)}(\eta-\xi)\int_0^1e^{\pi ix(\xi+\eta)}\ dx\\
&=\frac{1}{\pi i(\xi+\eta)}(e^{\pi i(\xi+\eta)}-1)\x_{[0,1)}(\eta-\xi)\\
&=e^{\frac{\pi i}{2}(\xi+\eta)}\frac{\sin(\frac{\pi}{2}(\xi+\eta))}{\frac{\pi}{2}(\xi+\eta)}\x_{[0,1)}(\eta-\xi),\ \ \ \xi+\eta\neq 0.
\end{align*}
If $\xi+\eta=0$, then 
\begin{align*}
K_{\p_1}(\xi,\eta)&=\x_{[0,1)}(\eta-\xi).
\end{align*}
Thus, using the sinc function, we express the kernel in the form
\begin{align*}
K_{\p_1}(\xi,\eta)&= 
e^{\frac{\pi i}{2}(\xi+\eta)}\sinc\left(\frac{\xi+\eta}{2}\right)\x_{[0,1)}(\eta-\xi).
\end{align*}
As, $\phi_m = \underset{i = 1}{\overset{m-1}{\times}} \phi_i$, $W(\p_m)=W(\p_1)^m$ since $W(\p\times\psi)=W(\p)W(\psi)$. It is known that if $T$ is an integral operator with kernel $K(\xi,\eta)$, then $T^2$ is an integral operator with kernel
\begin{align*}
\int K(\xi,y)K(y,\eta)\ dy.
\end{align*}
Thus, in general, $T^n$ is an integral operator with kernel
\begin{align*}
\int\cdots\int K(\xi,y_1)K(y_1,y_2)\cdots K(y_{n-2},y_{n-1})K(y_{n-1},\eta)\ dy_{n-1}\cdots dy_1.
\end{align*}
Hence, the kernel of {\it Weyl} transform of $\p_n$ is given by
\begin{align*}
K_{\p_n}(\xi,\eta) & =e^{\frac{\pi i}{2}(\xi+\eta)}\int_{y_1\in\R}\cdots\int_{y_{n-1}\in\R} e^{\pi i(y_1+\cdots+y_{n-1})}\sinc\left(\frac{\xi+y_1}{2}\right)\x_{[0,1)}(y_1-\xi)\times\\
&\qquad\sinc\left(\frac{y_1+y_2}{2}\right)\x_{[0,1)}(y_2-y_1)\cdots \sinc\left(\frac{y_{n-2}+y_{n-1}}{2}\right)\x_{[0,1)}(y_{n-1}-y_{n-2})\times\\
&\qquad\sinc\left(\frac{y_{n-1}+\eta}{2}\right)\x_{[0,1)}(\eta-y_{n-1})\ dy_{n-1}\cdots dy_1.
\end{align*}
First, we recursively express the kernel $K_{\p_n}$ in terms of $K_{\p_{n-1}}$. Observe that
\begin{align*}
K_{\p_2}(\xi,\eta)&=e^{\frac{\pi i}{2}(\xi+\eta)}\int_\R e^{\pi iy_1}\sinc\left(\frac{\xi+y_1}{2}\right)\sinc\left(\frac{y_1+\eta}{2}\right)\x_{[0,1)}(y_1-\xi)\x_{[0,1)}(\eta-y_1)\ dy_1\\
&=e^{\frac{\pi i}{2}(\xi+\eta)}\int_{y_1\in[\xi,\xi+1)\cap(\eta-1,\eta]}e^{\pi iy_1}\sinc\left(\frac{\xi+y_1}{2}\right)\sinc\left(\frac{y_1+\eta}{2}\right)\ dy_1.
\end{align*}
The above integral can be rewritten as
\begin{align*}
K_{\p_2}(\xi,\eta)&=e^{\frac{\pi i}{2}(\xi+\eta)}e^{\pi i\eta}\int_0^1 e^{-\pi iy_1}\sinc\left(\frac{\xi+\eta-y_1}{2}\right)\sinc\left(\frac{2\eta-y_1}{2}\right)\x_{(\eta-\xi-1,\eta-\xi]}(y_1)\ dy_1.
\end{align*}
Therefore,
\begin{align*}
K_{\p_n}(\xi,\eta) & =e^{\frac{\pi i}{2}(\xi+\eta)}\int_{y_{n-1}\in\R}\bigg(\int_{y_1\in\R}\cdots\int_{y_{n-2}\in\R} e^{\pi i(y_1+\cdots+y_{n-2})}\sinc\left(\frac{\xi+y_1}{2}\right)\x_{[0,1)}(y_1-\xi)\times\\
& \qquad\cdots\times\sinc\left(\frac{y_{n-2}+y_{n-1}}{2}\right)\x_{[0,1)}(y_{n-1}-y_{n-2})\ dy_{n-2}\cdots dy_1\bigg)\times\\
& \qquad \times e^{\pi iy_{n-1}}\sinc\left(\frac{y_{n-1}+\eta}{2}\right)\x_{[0,1)}(\eta-y_{n-1})\ dy_{n-1}\\
&=e^{\frac{\pi i}{2}(\xi+\eta)}\int_{y_{n-1}\in\R}e^{-\frac{\pi i}{2}(\xi+y_{n-1})}K_{\p_{n-1}}(\xi,y_{n-1})e^{\pi iy_{n-1}}\sinc\left(\frac{y_{n-1}+\eta}{2}\right)\x_{[0,1)}(\eta-y_{n-1})\ dy_{n-1}.
\end{align*}
Thus,
\begin{align*}
K_{\p_n}(\xi,\eta)=e^{\pi i\eta}\int_0^1e^{-\frac{\pi i}{2}y_{n-1}}\sinc\left(\frac{2\eta-y_{n-1}}{2}\right)K_{\p_{n-1}}(\xi,\eta-y_{n-1})\ dy_{n-1}.
\end{align*}
\end{proof}
\begin{proposition}
Given $f\in L^1(\R^2)$, we have the following identities:
\begin{align*}
\iint_{\R^2}f(x,y)\p_1(x,y)\ dx\,dy=\iint_Qf(x_1,y_1)\ dx_1\,dy_1, 
\end{align*}
\begin{align*}
\iint_{\R^2}f(x,y)\p_2(x,y)\ dx\,dy=\iint_{Q^2}e^{\pi i(x_2y_1-y_2x_1)}f(x_1+x_2,y_1+y_2)\ dx_1\,dy_1\,dx_2\,dy_2
\end{align*}
and
\begin{align*}
\iint_{\R^2}f(x,y)\p_3(x,y)\ dx\,dy=\iint_{Q^3}e^{\pi i[(x_2+x_3)(y_1+y_2)-(x_1+x_2)(y_2+y_3)]}f(x_1+x_2+x_3,y_1+y_2+y_3)\\ dx_1\,dy_1\,dx_2\,dy_2\,dx_3\,dy_3. 
\end{align*}
For $n\geq 4$,
\begin{align*}
\iint_{\R^2}f(x,y)\p_n(x,y)\ dx\,dy=\iint_{Q^n}e^{\pi i[(x_2+\cdots+x_n)(y_1+\cdots+y_{n-1})-(x_1+\cdots+x_{n-1})(y_2+\cdots+y_n)]}\ \times\\
e^{\pi i[\{y_2(x_3+\cdots+x_{n-1})-x_2(y_3+\cdots+y_{n-1})\}+\{y_3(x_4+\cdots+x_{n-1})-x_3(y_4+\cdots+y_{n-1})\}+\cdots+\{y_{n-3}(x_{n-2}+x_{n-1})-x_{n-3}(y_{n-2}+y_{n-1})\}]}\ \times\\
e^{\pi i(y_{n-2}x_{n-1}-x_{n-2}y_{n-1})}f(x_1+x_2+\cdots +x_n,y_1+y_2+\cdots +y_n)\ dx_1\,dy_1\,dx_2\,dy_2\cdots dx_n\,dy_n.
\end{align*}
\end{proposition}
\begin{proof}
We have
\begin{align*}
\iint_{\R^2}f(x,y)\p_1(x,y)\ dx\,dy&=\int_0^1\int_0^1f(x,y)\ dx\,dy =\iint_Qf(x_1,y_1)\ dx_1\,dy_1.
\end{align*}
Using \eqref{2} and then applying Fubini's theorem yields
\begin{align*}
\iint_{\R^2}f(x,y)\p_2(x,y)\ dx\,dy&=\iint_{\R^2}f(x,y)\bigg(\int_{u=0}^1\int_{v=0}^1\p_1(x-u,y-v)e^{\pi i(uy-vx)}\ dv\,du\bigg)dx\,dy\\
&=\iint_Q\bigg(\iint_{\R^2}f(x,y)\p_1(x-u,y-v)e^{\pi i(uy-vx)}\ dx\,dy\bigg)du\,dv.
\end{align*}
Applying the change of variables $x-u=r,\ y-v=s$, we obtain
\begin{align*}
\iint_{\R^2}f(x,y)\p_2(x,y)\ dx\,dy&=\iint_Q\bigg(\iint_{\R^2}f(r+u,s+v)\p_1(r,s)e^{\pi i(su-rv)}\ dr\,ds\bigg)du\,dv\\
&=\iint_Q\iint_Qf(r+u,s+v)e^{\pi i(su-rv)}\ dr\,ds\,du\,dv\\
&=\iint_{Q^2}e^{\pi i(x_2y_1-y_2x_1)}f(x_1+x_2,y_1+y_2)\ dx_1\,dy_1\,dx_2\,dy_2.
\end{align*}
Further, for $n=3$
\begin{align*}
\iint_{\R^2}f(x,y)\p_3(x,y)\ dx\,dy=\iint_{Q^3}e^{\pi i[(x_2+x_3)(y_1+y_2)-(x_1+x_2)(y_2+y_3)]}f(x_1+x_2+x_3,y_1+y_2+y_3)\\ dx_1\,dy_1\,dx_2\,dy_2\,dx_3\,dy_3. 
\end{align*}
Proceeding in this way, we get for $n\geq 4$,
\begin{align*}
\iint_{\R^2}f(x,y)\p_n(x,y)\ dx\,dy&=\iint_{\R^2}f(x,y)\bigg(\int_{x_n=0}^1\int_{y_n=0}^1\hspace{-2mm}\p_{n-1}(x-x_n,y-y_n)e^{\pi i(x_ny-y_nx)}\ dy_n\,dx_n\bigg)dx\,dy\\
&=\iint_Q\bigg(\iint_{\R^2}f(x,y)\p_{n-1}(x-x_n,y-y_n)e^{\pi i(x_ny-y_nx)}\ dx\,dy\bigg)dx_n\,dy_n\\
&=\iint_Q\bigg(\iint_{\R^2}T_{(-x_n,-y_n)}^tf(x,y)\p_{n-1}(x,y)\ dx\,dy\bigg)dx_n\,dy_n.\numberthis \label{6}
\end{align*}
As \eqref{6} holds also for $T_{(-x_n,-y_n)}^tf$, we obtain
\begin{align*}
\iint_{\R^2}T_{(-x_n,-y_n)}^tf(x,y)\p_{n-1}(x,y)\ dx\,dy=\iint_Q\bigg(\iint_{\R^2}T_{(-x_{n-1},-y_{n-1})}^tT_{(-x_n,-y_n)}^tf(x,y)\p_{n-2}&(x,y)\ dx\,dy\bigg)\\
&dx_{n-1}\,dy_{n-1}.
\end{align*}
Hence, \eqref{6} becomes
\begin{align*}
\iint_{\R^2}f(x,y)\p_n(x,y)\ dx\,dy=\iint_{Q^2}\bigg(\iint_{\R^2}T_{(-x_{n-1},-y_{n-1})}^tT_{(-x_n,-y_n)}^tf(x,y)\p_{n-2}&(x,y)\ dx\,dy\bigg)\\
&dx_n\,dy_n\,dx_{n-1}\,dy_{n-1}.
\end{align*}
Proceeding in this way, produces
\begin{align*}
\iint_{\R^2}f(x,y)\p_n(x,y)\ dx\,dy=\iint_{Q^{n-1}}\bigg(\iint_{\R^2}T_{(-x_2,-y_2)}^tT_{(-x_3,-y_3)}^t\cdots T&_{(-x_n,-y_n)}^tf(x,y)\p_1(x,y)\ dx\,dy\bigg)\\
&dx_2\,dy_2\,dx_3\,dy_3\cdots dx_n\,dy_n\\
=\iint_{Q^n}T_{(-x_2,-y_2)}^tT_{(-x_3,-y_3)}^t\cdots T_{(-x_n,-y_n)}^tf&(x_1,y_1)\ dx_1\,dy_1\,dx_2\,dy_2\,\cdots dx_n\,dy_n
\end{align*}
Now, composing all twisted translation operators, yields
\begin{align*}
\iint_{\R^2}f(x,y)\p_n(x,y)\ dx\,dy=\iint_{Q^n}e^{\pi i[(x_3y_2-y_3x_2)+\{x_4(y_2+y_3)-y_4(x_2+x_3)\}+\cdots +\{x_n(y_2+\cdots+ y_{n-1})-y_n(x_2+\cdots+x_{n-1})\}]}\\
\times\ T_{(-x_2-x_3-\cdots-x_n,-y_2-y_3-\cdots-y_n)}^tf(x_1,y_1)\ dx_1\,dy_1\,dx_2\,dy_2\cdots dx_n\,dy_n\\
=\iint_{Q^n}e^{\pi i[(x_2y_1-y_2x_1)+\{x_3(y_1+y_2)-y_3(x_1+x_2)\}+\cdots +\{x_n(y_1+\cdots+ y_{n-1})-y_n(x_1+\cdots+x_{n-1})\}]}\numberthis \label{7}\\
\times\ f(x_1+x_2+\cdots +x_n,y_1+y_2+\cdots +y_n)\ dx_1\,dy_1\,dx_2\,dy_2\cdots dx_n\,dy_n.
\end{align*}
Finally, we can rewrite \eqref{7} in the form
\begin{align*}
\iint_{\R^2}f(x,y)\p_n(x,y)\ dx\,dy=\iint_{Q^n}e^{\pi i[(x_2+\cdots+x_n)(y_1+\cdots+y_{n-1})-(x_1+\cdots+x_{n-1})(y_2+\cdots+y_n)]}\ \times\\
e^{\pi i[\{y_2(x_3+\cdots+x_{n-1})-x_2(y_3+\cdots+y_{n-1})\}+\{y_3(x_4+\cdots+x_{n-1})-x_3(y_4+\cdots+y_{n-1})\}+\cdots+\{y_{n-3}(x_{n-2}+x_{n-1})-x_{n-3}(y_{n-2}+y_{n-1})\}]}\ \times\\
e^{\pi i(y_{n-2}x_{n-1}-x_{n-2}y_{n-1})}f(x_1+x_2+\cdots +x_n,y_1+y_2+\cdots +y_n)\ dx_1\,dy_1\,dx_2\,dy_2\cdots dx_n\,dy_n.
\end{align*}
\end{proof}
\begin{proposition}
The twisted B-splines $\phi_1$ and $\phi_2$ satisfy the following partition-of-unity-line property:
\begin{align*}
\iint_{\R^2}\sum_{k,l\in\z}\t\p_1(x,y)\ dx\,dy&=1
\end{align*}
and
\begin{align*}
\iint_{\R^2}\sum_{k,l\in\z}\t\p_2(x,y)\ dx\,dy = C_{\p_2},
\end{align*}
where the constant $C_{\p_2}\approx 0.000160507/\pi^2$. Moreover, the functions $L_n$ defined by 
\[
\iint_Q L_n (u,v) du\,dv := \iint_{\R^2}\sum_{k,l\in\z}\t\p_{n+1}(x,y)\ dx\,dy,
\]
can be recursively computed in terms of their lower orders via
\[
L_{n+1}(u,v)=\iint_Qe^{\pi i(ut-vs)}L_n(u+s,v+t)\ ds\,dt.
\]
\end{proposition}
\begin{proof}
Recall that $\phi_1(x,y)=\x_{[0,1)}(x)\x_{[0,1)}(y)$. Then
\begin{align*}
\t\p_1(x,y)=e^{\pi i(lx-ky)}\x_{[k,k+1)}(x)\x_{[l,l+1)}(y).
\end{align*}
For each $x\in\R$, there exists a unique $p_x\in\z$ such that $p_x\leq x<p_x+1$. In fact, $p_x=\floor*{x}$. Thus
\begin{align*}
\sum_{k,l\in\z}\t\p_1(x,y)=e^{\pi i(\floor*{y}x-\floor*{x}y)}.
\end{align*}
Hence, 
\begin{align*}
\iint_{\R^2}\sum_{k,l\in\z}\t\p_1(x,y)\ dx\,dy&=\iint_{\R^2}e^{\pi i(\floor*{y}x-\floor*{x}y)}\ dx\,dy\\
&= \lim_{M\to\infty} \iint_{[-M,M]^2} e^{\pi i(\floor*{y}x-\floor*{x}y)}\ dx\,dy\\
&=\lim_{M\to\infty}\int_{x=-M}^{M}\int_{y=-M}^{M}e^{\pi i(\floor*{y}x-\floor*{x}y)}\ dy\,dx,\numberthis \label{8}
\end{align*}
where we used the \emph{Fubini-Tonelli} theorem to obtain the third line. But
\begin{align*}
\int_{x=-M}^{M}\int_{y=-M}^{M}e^{\pi i(\floor*{y}x-\floor*{x}y)}\ dxdy&=\sum_{k,l=-M}^{M-1}\int_{x=k}^{k+1}\int_{y=l}^{l+1}e^{\pi i(lx-ky)}\ dydx\\
&=A+B+C+1,\numberthis \label{9}
\end{align*}
say, where
\begin{align*}
A=\sum_{\substack{k,l=-M \\ (k,l)\neq(0,0)}}^{M-1}\int_{x=k}^{k+1}\int_{y=l}^{l+1}e^{\pi i(lx-ky)}\ dydx,
\end{align*}
\begin{align*}
\hspace{-6mm}& B=\sum_{\substack{k=-M \\ k\neq 0}}^{M-1}\int_{x=k}^{k+1}\int_{y=0}^1e^{-\pi iky}\ dydx,
\end{align*}
\begin{align*}
\hspace{-8mm}& C=\sum_{\substack{l=-M \\ l\neq 0}}^{M-1}\int_{x=0}^{1}\int_{y=l}^{l+1}e^{\pi ilx}\ dydx.
\end{align*}
Integrating $A$, we obtain
\begin{align*}
A&=\frac{1}{\pi^2}\sum_{\substack{k,l=-M \\ (k,l)\neq(0,0)}}^{M-1}\frac{1}{kl}(e^{\pi il}-1)(e^{-\pi ik}-1)\\
&=\frac{4}{\pi^2}\sum_{\substack{k,l=-M \\ (k,l)\neq(0,0),\ k,l\in 2\z+1}}^{M-1}\frac{1}{kl} =\frac{4}{\pi^2}L_M^2,
\end{align*}
where
\begin{align*}
L_M&=\sum_{\substack{k=-M \\ k\neq 0,\ k\in 2\z+1}}^{M-1}\frac{1}{k}=\begin{cases}
0, & \text{$M$ is even};\\
-\frac{1}{M}, & \text{$M$ is odd}.
\end{cases}
\end{align*}
Since $L_M\to 0$ as $M\to\infty$, we have $A\to 0$ as $M\to\infty$. On the other hand, integrating $B$ and $C$, we get
\begin{align*}
B+C&=-\sum_{\substack{k=-M \\ k\neq 0}}^{M-1}\frac{1}{\pi ik}(e^{-\pi ik}-1)+\sum_{\substack{l=-M \\ l\neq 0}}^{M-1}\frac{1}{\pi il}(e^{\pi il}-1)\\
&=-\sum_{\substack{k=-M \\ k\neq 0}}^{M-1}\frac{1}{\pi ik}((-1)^k-1)+\sum_{\substack{l=-M \\ l\neq 0}}^{M-1}\frac{1}{\pi il}((-1)^l-1)\\
&=0.\numberthis \label{10}
\end{align*}
Making use of \eqref{9} and \eqref{10} in \eqref{8} yields
\begin{align*}
\iint_{\R^2}\sum_{k,l\in\z}\t\p_1(x,y)\ dx\,dy&=1+\lim_{M\to\infty}A\\
&=1.
\end{align*}
Now, consider
\be\label{integrals}
\iint_{\R^2}\sum_{k,l\in\z}\t\p_2(x,y)\ dx\,dy = \frac{1}{\pi^2}\int_0^1\int_0^1 \sum_{p,q\in \z} \frac{e^{\pi i (u q -v p)}(e^{\pi i (u-p)}-1)(e^{-\pi i (v-q)}-1)}{(u-p)(v-q)}du\,dv.
\ee
First, we define the integrals
\[
\i(p,q) := \int_0^1\int_0^1\frac{e^{\pi i (u q -v p)}(e^{\pi i (u-p)}-1)(e^{-\pi i (v-q)}-1)}{(u-p)(v-q)}du\,dv,
\]
where $p,q\in \z$. In case that $\abs{p},\abs{q}\notin\{0,1\}$, the above integral evaluates to 
\begin{align*}
[\text{Ei}(i \pi p (q-1))-& \text{Ei}(i \pi p q)+\text{Ei}(i \pi (p+1) q)-\text{Ei}(i \pi (p+1)(q-1))]\times\\
& \left[\text{Ei}(-i \pi (p-1) q)-\text{Ei}(-i \pi p q)+\text{Ei}(-i \pi p (q+1)-\text{Ei}(-i \pi (p-1)(q+1)))\right]
\end{align*}
Setting
\begin{align*}
F_1(p,q) & := \text{Ei}(i \pi p (q-1))-\text{Ei}(i \pi p q)+\text{Ei}(i \pi (p+1) q)-\text{Ei}(i \pi (p+1)(q-1)),\\
F_2(p,q) &:= \text{Ei}(-i \pi (p-1) q)-\text{Ei}(-i \pi p q)+\text{Ei}(-i \pi p (q+1)-\text{Ei}(-i \pi (p-1)(q+1)),
\end{align*}
we note that $F_2(p,q) = \overline{F_1(q,p)}$. Thus,
\[
\i(p,q) := F_1(p,q) \overline{F_1(q,p)}, \qquad \abs{p},\abs{q}\notin\{0,1\}.
\]
Employing the identity $\Ei(- i x) = \overline{\Ei(i x)}$, for $x > 0$, we obtain the following relations: $\i(-p,-q) = \overline{\i(p,q)}$, $\i(-p,q) = \overline{\i(q,-p)}$, and $\i(p,-q) = \overline{\i(-q,p)}$.

For $\abs{p},\abs{q}\in\{0,1\}$, the integrals $\i(p,q)$ evaluate to
\begin{align*}
&\i(0,0) = (\gamma - \Ci(x) + \log\pi)^2 + \Si^2(x);\\
&\i(1,0) = (\text{Ei}(-2 i \pi )+\text{Ei}(-i \pi )+\log 2) (\text{Ci}(\pi ) -i \text{Si}(\pi)-\gamma -\log \pi );\\
&\i(0,1)  = \overline{\i(1,0)};\\
&\i(1,1)  = (\overline{\text{Ei}(i \pi )-\text{Ei}(2 i \pi) }+\log 2) (\text{Ei}(i \pi )-\text{Ei}(2 i \pi)+\log 2);\\
&\i(0,-1) = \i(1,0);\\
&\i(-1,0) = \i(0,1) = \overline{\i(1,0)};\\
&\i(-1,-1) = \i(1,1);\\
&\i(1,-1) = (\text{Ei}(-i \pi )-2 \text{Ei}(-2 i \pi )+\text{Ei}(-4 i \pi )) (\text{Ci}(\pi )+i \text{Si}(\pi )-\gamma -\log \pi);\\
&\i(-1,1) = \overline{\i(1,-1)},
\end{align*}
and 
\begin{align*}
&\i(p,0) = \left(\text{Ei}(-i \pi (p-1))-\text{Ei}(-i \pi p)+\log \left(\frac{p}{p-1}\right)\right) \left(-\text{Ei}(-i \pi p)+\text{Ei}(-i \pi (p+1))+\log \left(\frac{p}{p+1}\right)\right);\\
&\i(p,1) =  \bigg(-\text{Ei}(-i \pi (p-1))+\text{Ei}(-2 i \pi (p-1))+\text{Ei}(-i \pi p)-\text{Ei}(-2 i \pi p)\bigg)\times\\ 
&\hspace*{2.5cm}\left(\text{Ei}(i \pi p)-\text{Ei}(i \pi (p+1))+\log \left(\frac{1+p}{p}\right)\right);\\
&\i(p,-1) = \bigg(\text{Ei}(-i \pi p)-\text{Ei}(-2 i \pi p)-\text{Ei}(-i \pi (p+1))+\text{Ei}(-2 i  \pi(p+1))\bigg)\times\\
&\hspace*{2.5cm} \left(-\text{Ei}(i \pi (p-1))+\text{Ei}(i \pi p)+\log\left(\frac{p-1}{p}\right)\right);\\
&\i(0,q) = \overline{\i(q,0)};\\
&\i(1,q) = \overline{\i(q,1)};\\
&\i(-1,q) = \overline{\i(q,-1)}.
\end{align*}

Finally, for $p\geq 2$ and $q\leq -2$, 
\begin{align*}
\i(p,q) & = (\text{Ei}(i \pi (p-1)(q-1))-\text{Ei}(i \pi p(q-1))-\text{Ei}(i \pi (p-1)q)+\text{Ei}(i \pi p q)) \times \\
& \qquad(\text{Ei}(-i \pi p q)-\text{Ei}(-i \pi (p+1)q)-\text{Ei}(-i \pi p(q+1))+\text{Ei}(-i \pi (p+1)(q+1))).
\end{align*}
Next, we will show that $\sum\limits_{p,q\in\z} \i(p,q)$ converges as an iterated sum which allows the interchange of integrals and sums in \eqref{integrals}.

For this purpose, we need to write $\sum\limits_{p,q\in\z} \i(p,q)$ as follows.
\begin{align*}\label{sums}
\sum\limits_{p,q\in\z} \i(p,q) & = \sum_{\abs{p},\abs{q}\leq 1} \i(p,q) + \sum_{\abs{p}\geq 2} \i(p,0) + \sum_{\abs{p}\geq 2} \i(p,1) + \sum_{\abs{p}\geq 2} \i(p,-1)\numberthis\\
& \quad + \sum_{\abs{q}\geq 2} \i(0,q)  + \sum_{\abs{q}\geq 2} \i(1,q) + \sum_{\abs{q}\geq 2} \i(-1,q)\\
& \quad + \sum_{p\geq 2, q\leq -2} \i(p,q) + \sum_{p\leq 2, q\geq 2} \i(p,q)+ \sum_{\abs{p},\abs{q}\geq 2} \i(p,q) .
\end{align*}

Using the afore-mentioned identities between the integrals $\i(p,q)$, the above sums reduce to
\begin{align*}
\sum\limits_{p,q\in\z} \i(p,q) & = \sum_{\abs{p},\abs{q}\leq 1} \i(p,q) + 2\sum_{{p}\geq 2} 2\re{\i(p,0)} + 2\sum_{{p}\geq 2} 2\re{\i(p,1)} + 2\sum_{{p}\geq 2} 2\re{\i(p,-1)}\\
& \quad + \sum_{p\geq 2, q\leq -2} 2\re{\i(p,q)} + \sum_{{p},{q}\geq 2} 2\re{\i(p,q)}.
\end{align*}

Recalling that $\Ei(i x) = \Ci(x) + i(\Si(x) -\frac\pi2)$ \cite{leb}, we express $\i(p,0)$, $\i(p,1)$, $\i(p,-1)$, and $\i(p,q)$ ($p\geq 2$, $q\leq -2$) in terms of $\Ci$ and $\Si$, and take the real part. This yields,
\begin{align*}
\re\i(p,0) &= \left(\Ci(\pi(p-1))-\Ci(\pi p)+\log\left(\frac{p}{p-1}\right)\right)\left(\Ci(\pi(p+1))-\Ci(\pi p)+\log\left(\frac{p}{p+1}\right)\right)\\
& \quad -\bigg(\Si(\pi p)-\Si(\pi(p-1))\bigg)\bigg(\Si(\pi p)-\Si(\pi(p+1))\bigg);\\
\re\i(p,1) &= \bigg(\Ci(\pi p) - \Ci(\pi(p-1))+\Ci(2\pi(p-1))-\Ci(2\pi p)\bigg)\left(\Ci(\pi p)-\Ci(\pi(p+1))+\log\left(\frac{p+1}{p}\right)\right)\\
& \quad - \bigg(\Si(\pi p) - \Si(\pi(p-1))+\Si(2\pi(p-1))-\Si(2\pi p)\bigg)\bigg(\Si(\pi p)-\Si(\pi(p+1))\bigg);\\
\re\i(p,-1) &= \bigg(\Ci(\pi p) - \Ci(2\pi p) + \Ci(2\pi(p+1))-\Ci(\pi(p+1))\bigg)\left(\Ci(\pi p) - \Ci(\pi(p-1))+\log\left(\frac{p-1}{p}\right)\right)\\
& \quad - \bigg(\Si(\pi p) - \Si(2\pi p) + \Si(2\pi(p+1))-\Si(\pi(p+1))\bigg)\bigg(\Si(\pi p) - \Si(\pi(p-1))\bigg);\\
\re\i(p,q) &= \bigg(\Ci(\pi(p-1)(q-1)) - \Ci(\pi q(q-1)) +\Ci(\pi p q) - \Ci(\pi(p-1)q)\bigg)\times\\
& \qquad\bigg(\Ci(\pi(p+1)(q+1)) - \Ci(\pi (p+1)q) +\Ci(\pi p q) - \Ci(\pi p(q+1))\bigg)\\
& \quad + \bigg(\Si(\pi(p-1)(q-1)) - \Si(\pi q(q-1)) +\Si(\pi p q) - \Si(\pi(p-1)q)\bigg)\times\\
& \qquad\bigg(\Si(\pi(p+1)(q+1)) - \Si(\pi (p+1)q) +\Si(\pi p q) - \Si(\pi p(q+1))\bigg)
\end{align*}
Similarly, we write $F_1(p,q)$ in terms of $\Ci$ and $\Si$ and obtain
\begin{align*}
F_1(p,q) &= \left[\Ci(\pi p (q-1))-\Ci(\pi p q) +\Ci(\pi(p+1)q)-\Ci(\pi(p+1)(q-1))\right] +\\
& \quad + i \left[\Si(\pi p (q-1))-\Si(\pi p q) +\Si(\pi(p+1)q)-\Si(\pi(p+1)(q-1))\right]\\
& \quad =: u(p,q) + i v(p,q).
\end{align*}
The sum $\sum\limits_{{p},{q}\geq 2} 2\re{\i(p,q)}$ has then the form
\[
\sum_{{p},{q}\geq 2} 2\re{\i(p,q)} = \sum_{{p},{q}\geq 2} 2\big[u(p,q) u(q,p) + v(p,q)v(q,p)\big].
\]

The main idea to show convergence of the infinite sums in \eqref{sums} is to estimate the $\Ci$ and $\Si$ terms. To this end, we use the following asymptotic representations of these two functions found in \cite{leb}.
\begin{subequations}
\begin{align}
\Ci (x) & = \frac{\sin x}{x}\,P(x) - \frac{\cos x}{x}\,Q(x)\label{asym1},\\
\frac\pi2 - \Si (x) & = \frac{\cos x}{x}\,P(x) + \frac{\sin x}{x}\,Q(x)\label{asym2},
\end{align}
\end{subequations}
where
\begin{subequations}
\begin{align}
P(x) &= \sum_{k=0}^n \frac{(-1)^k (2k)!}{x^{2k}} + \cO(|x|^{-2n-2}),\label{PQ1}\\
Q(x) &= \sum_{k=0}^n \frac{(-1)^k (2k+1)!}{x^{2k+1}} + \cO(|x|^{-2n-3})\label{PQ2}.
\end{align}
\end{subequations}

In addition, we need estimates on the logarithmic terms appearing in the above formulas for the real parts of $\i(p,0)$, $\i(p,1)$, and $\i(p,-1)$. We summarize these estimates in the following lemma.
\begin{lemma}
Let $p\geq 2$. Then
\begin{enumerate}
\item[\emph{(a)}] $\log\left(\dfrac{p}{p-1}\right)\leq \dfrac{1}{p-1}\leq\dfrac{2}{p}$;
\item[\emph{(b)}] $\abs{\log\left(\dfrac{p}{p+1}\right)}\leq \dfrac{1}{p-1}$;
\item[\emph{(c)}] $\log\left(\dfrac{p}{p-1}\right)\abs{\log\left(\dfrac{p}{p+1}\right)}\leq \dfrac{1}{(p-1)^2}$.
\end{enumerate}
\end{lemma}
\begin{proof}
These statements follow immediately from rewriting the inequalities in terms of the exponential function and using its power series expansion.
\end{proof}

Using \eqref{asym1} and \eqref{asym2}, we rewrite $\re \i(p,0)$ in the following way.
\begin{align*}
\re \i(p,0) & = \left[\frac{(-1)^p}{\pi(p-1)}\,Q(\pi(p-1)) - \frac{(-1)^{p+1}}{\pi p}\,Q(\pi p) + \log\left(\frac{p}{p-1}\right)\right]\times\\
& \qquad \left[\frac{(-1)^p}{\pi(p +1)}\,Q(\pi(p+1)) - \frac{(-1)^{p+1}}{\pi p}\,Q(\pi p) + \log\left(\frac{p}{p+1}\right)\right]\\
& - \left[\frac{(-1)^{p-1}}{\pi(p-1)}\,P(\pi(p-1)) - \frac{(-1)^{p}}{\pi p}\,P(\pi p)\right]\left[\frac{(-1)^{p+1}}{\pi(p+1)}\,P(\pi(p+1)) - \frac{(-1)^{p}}{\pi p}\,P(\pi p)\right]
\end{align*}

Therefore, 
\begin{align*}
\abs{\re \i(p,0)} &\leq \left[\frac{\abs{Q(\pi(p-1))}}{\pi(p-1)} + \frac{\abs{Q(\pi p)}}{\pi p} + \log\left(\frac{p}{p-1}\right)\right] \left[\frac{\abs{Q(\pi(p+1))}}{\pi(p+1)} + \frac{\abs{Q(\pi p)}}{\pi p} + \abs{\log\left(\frac{p}{p+1}}\right)\right]\\
& \quad + \left[\frac{\abs{P(\pi(p-1))}}{\pi(p-1)} + \frac{P(\pi p)}{\pi p}\right]\left[\frac{\abs{P(\pi(p+1))}}{\pi(p+1)} + \frac{\abs{P(\pi p)}}{\pi p}\,\right]
\end{align*}

By \eqref{PQ1} and \eqref{PQ2}, there exist positive constants $c_1, \ldots, c_4$ such that
\begin{align*}
\abs{\re \i(p,0)} &\leq \left[\frac{c_1}{p-1} + \log\left(\frac{p}{p-1}\right)\right]\left[\frac{c_2}{p} + \abs{\log\left(\frac{p}{p+1}}\right)\right] + \left[\frac{c_3}{p-1}\right] \left[\frac{c_4}{p}\right]\\
& \leq \frac{c_5}{p(p-1)} +\frac{c_6}{p^2} + \frac{c_7}{(p-1)^2} \leq \frac{c_8}{(p-1)^2}.
\end{align*}

Here, we used the statements in the lemma to obtain the last line.

In a similar fashion, one shows that $\abs{\re \i(p,\pm1)}$ is bounded above by $\frac{c_\pm}{(p-1)^2}$, for positive constants $c_\pm$. Therefore, the single infinite sums in \eqref{sums} are all bounded above by a convergent series of the form $c\,\sum\limits_{p=2}^\infty (p-1)^{-2}$, $c > 0$, and are thus themselves convergent.

It remains to establish the convergence of the double infinite sums $\re \i(p,q)$  $(p\geq 2, q\leq -2)$ and $\re \i(p,q)$ $(p,q\geq 2)$. This, however, can be done by applying essentially the same arguments as above for the single infinite sums. One thus obtains
\begin{align*}
\abs{\re \i(p,q)} \leq \frac{c}{(p-1)^2(q-1)^2}, \quad c > 0,
\end{align*}
for either range of parameters $p$ and $q$. Hence, the double infinite series in \eqref{sums} are also convergent.

A numerical evaluation of the first 1,000 terms in $\sum\limits_{p,q\in\z} \i(p,q)$ gives a value of approximately 0.000160507. Hence, 
\[
\iint_{\R^2}\sum_{k,l\in\z}\t\p_2(x,y)\ dx\,dy = C_{\p_2},
\]
where $C_{\p_2}\approx 0.000160507/\pi^2$.
\vskip 10pt

To verify the recursive formulas for the functions $L_n$, note that for $n\geq 1$, we have that
\begin{align*}
\sum_{k,l\in\z}\t\p_{n+1}(x,y)&=\sum_{k,l\in\z}e^{\pi i(lx-ky)}\p_{n+1}(x-k,y-l)\\
&=\sum_{k,l\in\z}e^{\pi i(lx-ky)}\int_{u=0}^1\int_{v=0}^1\p_n(x-k-u,y-l-v)e^{\pi i(u(y-l)-v(x-k))}\ dv\,du\\
&=\iint_Qe^{\pi i(uy-vx)}\sum_{k,l\in\z}e^{\pi i(l(x-u)-k(y-v))}\p_n(x-u-k,y-v-l)\ du\,dv\\
&=\iint_Qe^{\pi i(uy-vx)}\sum_{k,l\in\z}\t\p_n(x-u,y-v)\ du\,dv,
\end{align*}
and, therefore,
\begin{align*}
\iint_{\R^2}\sum_{k,l\in\z}\t\p_{n+1}(x,y)\ dx\,dy&=\iint_Q\bigg(\iint_{\R^2}e^{\pi i(uy-vx)}\sum_{k,l\in\z}\t\p_n(x-u,y-v)\ dx\,dy\bigg)du\,dv\\
&=:\iint_Q L_n(u,v)\ du\,dv.
\end{align*}
Thus, for $n\geq 1$,
\begin{align*}
L_{n+1}(u,v)&=\iint_{\R^2}e^{\pi i(uy-vx)}\sum_{k,l\in\z}\t\p_{n+1}(x-u,y-v)\ dx\,dy\\
&=\iint_{\R^2}e^{\pi i(uy-vx)}\sum_{k,l\in\z}e^{\pi i(l(x-u)-k(y-v))}\p_{n+1}(x-u-k,y-v-l)\ dx\,dy\\
&=\iint_{\R^2}e^{\pi i(uy-vx)}\sum_{k,l\in\z}e^{\pi i(l(x-u)-k(y-v))}\bigg(\iint_Q\p_n(x-u-k-s,y-v-l-t)\ \times\\
&\hspace{9 cm}e^{\pi i(s(y-v-l)-t(x-u-k))}\ ds\,dt\bigg)dx\,dy.
\end{align*}
Now, applying Fubini's theorem yields
\begin{align*}
L_{n+1}(u,v)&=\iint_Q\sum_{k,l\in\z}\bigg(\iint_{\R^2}e^{\pi i(uy-vx)}e^{\pi i(s(y-v)-t(x-u))}e^{\pi i(l(x-u-s)-k(y-v-t))}\ \times\\
&\hspace{4 cm}\p_n(x-u-k-s,y-v-l-t)\ dx\,dy\bigg)dsdt\\
&=\iint_Q\sum_{k,l\in\z}\bigg(\iint_{\R^2}e^{\pi i(uy-vx)}e^{\pi i(s(y-v)-t(x-u))}\t\p_n(x-u-s,y-v-t)\ dx\,dy\bigg)ds\,dt\\
&=\iint_Qe^{\pi i(ut-vs)}\bigg(\iint_{\R^2}e^{\pi i((u+s)y-(v+t)x)}\sum_{k,l\in\z}\t\p_n(x-u-s,y-v-t)\ dx\,dy\bigg)ds\,dt\\
&=\iint_Qe^{\pi i(ut-vs)}L_n(u+s,v+t)\ ds\,dt.
\end{align*}
\end{proof}
\section{System of Twisted Translates of $\p_n$}
First, we prove that $\{\t\p_1:k,l\in\z\}$ forms an orthonormal system in $L^2(\R^2)$. Consider 
\begin{align*}
\big\langle\t\p_1,T_{(k^\prime,l^\prime)}^t\p_1\big\rangle_{L^2(\R^2)}&=\iint_{\R^2}\t\p_1(x,y)\ \overline{T_{(k^\prime,l^\prime)}^t\p_1(x,y)}\ dxdy\\
&=\iint_{\R^2}e^{\pi i(x(l-l^\prime)-y(k-k^\prime))}\p_1(x-k,y-l)\ \overline{\p_1(x-k^\prime,y-l^\prime)}\ dxdy.
\end{align*}
Applying a change of variables, we get
\begin{align*}
\big\langle\t\p_1,T_{(k^\prime,l^\prime)}^t\p_1\big\rangle_{L^2(\R^2)}&=\iint_{\R^2}e^{\pi i(k^\prime l-l^\prime k)}e^{\pi i(x(l-l^\prime)-y(k-k^\prime))}\p_1(x,y)\ \overline{\p_1(x+k-k^\prime,y+l-l^\prime)}\ dxdy\\
&=(-1)^{k^\prime l-l^\prime k}\int_{x=0}^1\int_{y=0}^1e^{\pi i(x(l-l^\prime)-y(k-k^\prime))}\x_{[0,1)}(x+k-k^\prime)\ \times\\
&\hspace{9 cm} \x_{[0,1)}(y+l-l^\prime)dydx\\
&=(-1)^{k^\prime l-l^\prime k}\int_{x=0}^1\int_{y=0}^1e^{\pi i(x(l-l^\prime)-y(k-k^\prime))}\x_{[k^\prime-k,k^\prime-k+1)}(x)\ \times\\
&\hspace{9 cm} \x_{[l^\prime-l,l^\prime-l+1)}(y)\ dydx\\
&=(-1)^{k^\prime l-l^\prime k}\ \delta_{k,k^\prime}\ \delta_{l,l^\prime},
\end{align*}
proving that $\{\t\p_1:k,l\in\z\}$ forms an orthonormal system.
\begin{theorem}\label{c}
Let $\p_2$ denote the second order twisted $B$-spline. Then $\{\t\p_2:k,l\in\z\}$ is a {\it Riesz} basis for the twisted shift-invariant space $V^t(\p_2)$.
\end{theorem}
\begin{proof}
Recall that $V^t(\p_2)=\overline{\Span\ \{\t\p_2:k,l\in\z\}}$. Hence is enough to prove that the collection $\{\t\p_2:k,l\in\z\}$ is a {\it Riesz} sequence. Let $G$ denote the Gramian associated with the system $\{\t\p_2:k,l\in\z\}$ and let $\{c_{k,l}\}\in\ell^2(\z^2)$. We need to prove that there exist $A,B>0$ such that
\begin{align*}
A\|\{c_{k,l}\}\|_{\ell^2(\z^2)}^2\leq \big\langle G\{c_{k,l}\},\{c_{k,l}\}\big\rangle_{\ell^2(\z^2)}\leq B\|\{c_{k,l}\}\|_{\ell^2(\z^2)}^2.
\end{align*}
Consider
\begin{align*}
\big\langle G\{c_{k,l}\},\{c_{k,l}\}\big\rangle_{\ell^2(\z^2)}&=\sum_{k,k^\prime,l,l^\prime\in\z}c_{k,l}\ \overline{c_{k^\prime,l^\prime}}\ \big\langle\t\p_2,T_{(k^\prime,l^\prime)}^t\p_2\big\rangle_{L^2(\R^2)}\\
&=\sum_{k,k^\prime,l,l^\prime\in\z}c_{k,l}\overline{c_{k^\prime,l^\prime}}\iint_{\R^2}e^{\pi i(x(l-l^\prime)-y(k-k^\prime))}\p_2(x-k,y-l)\ \overline{\p_2(x-k^\prime,y-l^\prime)}\ dxdy\\
&=\sum_{k,k^\prime,l,l^\prime\in\z}c_{k,l}\overline{c_{k^\prime,l^\prime}}(-1)^{k^\prime l-l^\prime k}\iint_{\R^2}e^{\pi i(x(l-l^\prime)-y(k-k^\prime))}\p_2(x+k^\prime-k,y+l^\prime-l)\ \times\\
&\hspace{11 cm}\overline{\p_2(x,y)}\ dxdy,
\end{align*}
after applying a change of variables. As the support of $\p_2=[0,2]\times[0,2]$, it can be easily shown that in the sum for $k,l\in\z$ only three values of $k$ and $l$ survive. Namely,
\begin{align*}
&k=k^\prime-1,\ k^\prime,\ k^\prime+1,\\
&l=l^\prime-1,\ l^\prime,\ l^\prime+1.
\end{align*}
Thus, we can write
\begin{align*}
\big\langle G\{c_{k,l}\},\{c_{k,l}\}\big\rangle_{\ell^2(\z^2)}&=S_1+S_2+\cdots+S_9,
\end{align*}
depending upon the above choices of $k$ and $l$.
Now
\begin{align*}
S_1&=\sum_{k^\prime,l^\prime\in\z}c_{k^\prime-1,l^\prime-1}\ \overline{c_{k^\prime,l^\prime}}\ e^{\pi i(l^\prime-k^\prime)}\int_{x=0}^{2}\int_{y=0}^{2}e^{\pi i(y-x)}\p_2(x+1,y+1)\p_2(x,y)\ dydx.
\end{align*}
In the integral limits, only $(x,y)\in(0,1]\times(0,1]$ contributes and $(x,y)\in(0,1]\times(1,2]$, $(x,y)\in(1,2]\times(0,1]$ and $(x,y)\in(1,2)\times(1,2)$ can not contribute using the supports of intersection of $\p_2(\cdot+1,\cdot+1)$ and $\p_2(\cdot,\cdot)$. Thus, from \eqref{1}, we obtain
\begin{align*}
S_1&=\sum_{k^\prime,l^\prime\in\z}c_{k^\prime-1,l^\prime-1}\ \overline{c_{k^\prime,l^\prime}}\ e^{\pi i(l^\prime-k^\prime)}\frac{4}{\pi^4}\int_{x=0}^1\int_{y=0}^1\frac{e^{\pi i(y-x)}}{xy(x+1)(y+1)}[\cos(\pi(x-y))-\cos(\pi(1-xy))]\times\\
&\hspace{12 cm} (1-\cos(\pi xy))\ dy\,dx.
\end{align*}
In a similar way, we get that
\begin{align*}
S_2&=\sum_{k^\prime,l^\prime\in\z}\overline{c_{k^\prime-1,l^\prime-1}}\ c_{k^\prime,l^\prime}\ e^{-\pi i(l^\prime-k^\prime)}\int_{x=0}^{1}\int_{y=0}^{1}e^{-\pi i(y-x)}\p_2(x+1,y+1)\p_2(x,y)\ dy\,dx\\
&=\overline{S_1},
\end{align*}
\begin{align*}
S_3&=\sum_{k^\prime,l^\prime\in\z}c_{k^\prime-1,l^\prime}\ \overline{c_{k^\prime,l^\prime}}\ e^{\pi il^\prime}\iint_{\R^2}e^{\pi iy}\p_2(x+1,y)\ \overline{\p_2(x,y)}\ dx\,dy\\
&=\sum_{k^\prime,l^\prime\in\z}c_{k^\prime-1,l^\prime}\ \overline{c_{k^\prime,l^\prime}}\ e^{\pi il^\prime}\frac{4}{\pi^4}\bigg(\int_{x=0}^1\int_{y=0}^1\frac{e^{\pi iy}}{x(x+1)y^2}(\cos(\pi xy)-\cos(\pi y))(1-\cos(\pi xy))\ dy\,dx\ +\\
&\hspace{2 mm}\int_{x=0}^1\int_{y=1}^2\frac{e^{\pi iy}}{x(x+1)y^2}[\cos(\pi (x-y+1))-\cos(x-xy+1)][\cos(\pi x(y-1))-\cos(\pi x)]\ dy\,dx\bigg),
\end{align*}
\begin{align*}
S_4&=\sum_{k^\prime,l^\prime\in\z}c_{k^\prime+1,l^\prime}\ \overline{c_{k^\prime,l^\prime}}\ e^{-\pi il^\prime}\iint_{\R^2}e^{-\pi iy}\p_2(x-1,y)\ \overline{\p_2(x,y)}\ dx\,dy\\
&=\sum_{k^\prime,l^\prime\in\z}c_{k^\prime,l^\prime}\ \overline{c_{k^\prime-1,l^\prime}}\ e^{-\pi il^\prime}\iint_{\R^2}e^{-\pi iy}\p_2(x,y)\ \overline{\p_2(x+1,y)}\ dx\,dy\\
&=\overline{S_3},
\end{align*}
\begin{align*}
S_5&=\sum_{k^\prime,l^\prime\in\z}c_{k^\prime,l^\prime-1}\ \overline{c_{k^\prime,l^\prime}}\ e^{-\pi ik^\prime}\iint_{\R^2}e^{-\pi ix}\p_2(x,y+1)\ \overline{\p_2(x,y)}\ dx\,dy\\
&=\sum_{k^\prime,l^\prime\in\z}c_{k^\prime,l^\prime-1}\ \overline{c_{k^\prime,l^\prime}}\ e^{-\pi ik^\prime}\frac{4}{\pi^4}\bigg(\int_{x=0}^1\int_{y=0}^1\frac{e^{-\pi ix}}{x^2y(y+1)}[\cos(\pi xy)-\cos(\pi x)][1-\cos(\pi xy)]\ dy\,dx\ +\\
&\int_{x=1}^2\int_{y=0}^1\frac{e^{-\pi ix}}{x^2y(y+1)}[\cos(\pi (x-y-1))-\cos(\pi(y-xy+1))][\cos(\pi (x-1)y)-\cos(\pi y)]\ dy\,dx\bigg).
\end{align*}
Notice that the integrand in $S_5$ is same as the integrand in $\overline{S_3}$ by interchanging the roles of $x$ and $y$.
\begin{align*}
S_6&=\sum_{k^\prime,l^\prime\in\z}c_{k^\prime,l^\prime+1}\ \overline{c_{k^\prime,l^\prime}}\ e^{\pi ik^\prime}\iint_{\R^2}e^{\pi ix}\p_2(x,y-1)\ \overline{\p_2(x,y)}\ dx\,dy\\
&=\sum_{k^\prime,l^\prime\in\z}c_{k^\prime,l^\prime}\ \overline{c_{k^\prime,l^\prime-1}}\ e^{\pi ik^\prime}\iint_{\R^2}e^{\pi ix}\p_2(x,y)\ \overline{\p_2(x,y+1)}\ dx\,dy\\
&=\overline{S_5}.
\end{align*}
Notice that the integrand in $S_6$ is same as the integrand in $S_3$ with $x$ and $y$ interchanged.
\begin{align*}
S_7&=\sum_{k^\prime,l^\prime\in\z}c_{k^\prime-1,l^\prime+1}\ \overline{c_{k^\prime,l^\prime}}\ e^{\pi i(k^\prime+l^\prime)}\iint_{\R^2}e^{\pi i(x+y)}\p_2(x+1,y-1)\ \overline{\p_2(x,y)}\ dx\,dy\\
&=\sum_{k^\prime,l^\prime\in\z}c_{k^\prime-1,l^\prime+1}\ \overline{c_{k^\prime,l^\prime}}\ e^{\pi i(k^\prime+l^\prime)}\frac{4}{\pi^4}\int_{x=0}^1\int_{y=1}^2\frac{e^{\pi i(x+y)}}{x(x+1)y(y-1)}[\cos(\pi x(y-1))-\cos(\pi(y-1))]\ \times\\
&\hspace{10 cm}[\cos(\pi x(y-1))-\cos(\pi x)]\ dy\,dx.
\end{align*}
\begin{align*}
S_8&=\sum_{k^\prime,l^\prime\in\z}c_{k^\prime+1,l^\prime-1}\ \overline{c_{k^\prime,l^\prime}}\ e^{-\pi i(k^\prime+l^\prime)}\iint_{\R^2}e^{-\pi i(x+y)}\p_2(x-1,y+1)\ \overline{\p_2(x,y)}\ dx\,dy\\
&=\sum_{k^\prime,l^\prime\in\z}c_{k^\prime,l^\prime}\ \overline{c_{k^\prime-1,l^\prime+1}}\ e^{-\pi i(k^\prime+l^\prime)}\iint_{\R^2}e^{-\pi i(x+y)}\p_2(x,y)\ \overline{\p_2(x+1,y-1)}\ dx\,dy\\
&=\overline{S_7}.
\end{align*}
\begin{align*}
S_9&=\sum_{k^\prime,l^\prime\in\z}|c_{k^\prime,l^\prime}|^2\iint_{\R^2}|\p_2(x,y)|^2\ dxdy\\
&=\sum_{k^\prime,l^\prime\in\z}|c_{k^\prime,l^\prime}|^2\frac{4}{\pi^4}\bigg(\int_{x=0}^1\int_{y=0}^1\frac{1}{x^2y^2}[1-\cos(\pi xy)]^2\ dy\,dx\ +\ \int_{x=0}^1\int_{y=1}^2\frac{1}{x^2y^2}[\cos(\pi x(y-1))-\\
&\hspace{4 cm}\cos(\pi x)]^2\ dy\,dx\ +\ \int_{x=1}^2\int_{y=0}^1\frac{1}{x^2y^2}[\cos(\pi (x-1)y)-\cos(\pi y)]^2\ dy\,dx\ +\\
&\hspace{5 cm}\int_{x=1}^2\int_{y=1}^2\frac{1}{x^2y^2}[\cos(\pi (x-y))-\cos(\pi (x+y-xy))]^2\ dy\,dx\bigg).
\end{align*}
Thus,
\begin{align*}
\big\langle G\{c_{k,l}\},\{c_{k,l}\}\big\rangle_{\ell^2(\z^2)}&=\sum_{k^\prime,l^\prime\in\z}c_{k^\prime-1,l^\prime-1}\ \overline{c_{k^\prime,l^\prime}}\ e^{\pi i(l^\prime-k^\prime)}\frac{1}{\pi^4}\ I_1\ +\ \overline{S_1}\ +\ \sum_{k^\prime,l^\prime\in\z}c_{k^\prime-1,l^\prime}\ \overline{c_{k^\prime,l^\prime}}\ e^{\pi il^\prime}\frac{1}{\pi^4}\ I_3\ +\\
&\overline{S_3}\ +\ \sum_{k^\prime,l^\prime\in\z}c_{k^\prime,l^\prime-1}\ \overline{c_{k^\prime,l^\prime}}\ e^{-\pi ik^\prime}\frac{1}{\pi^4}\ I_5\ +\ \overline{S_5}\ +\ \sum_{k^\prime,l^\prime\in\z}c_{k^\prime-1,l^\prime+1}\ \overline{c_{k^\prime,l^\prime}}\ e^{\pi i(k^\prime+l^\prime)}\frac{1}{\pi^4}\ I_7\\
&\hspace{7 cm}+\ \overline{S_7}\ +\ \sum_{k^\prime,l^\prime\in\z}|c_{k^\prime,l^\prime}|^2\frac{1}{\pi^4}\ I_9.
\end{align*}
As the values of the integral $I_5$ and $I_6$ are same as the integral in $\overline{S_3}$ and $I_3$ respectively, it is enough if one evaluates the integrals $I_1,I_3,I_7$ and $I_9$. These integrals are numerically computed and the values are the following:
\begin{align*}
I_1&=0.531003-i\ 0.467628\\
I_3&=-1.97877+i\ 0.56791\\
I_7&=0.906616-i\ 0.390131\\
I_9&=14.3661.
\end{align*}
Hence $|I_1|=0.707559,|I_3|=2.05865$ and $|I_7|=0.986993$. Now
\begin{align*}
\big\langle G\{c_{k,l}\},\{c_{k,l}\}\big\rangle_{\ell^2(\z^2)}&\leq \frac{1}{\pi^4}\|\{c_{k,l}\}\|_{\ell^2(\z^2)}^2(2|I_1|+4|I_3|+2|I_7|+I_9)\\
&=\frac{1}{\pi^4}\|\{c_{k,l}\}\|_{\ell^2(\z^2)}^2(11.6237+14.3661)\\
&=\frac{25.9898}{\pi^4}\|\{c_{k,l}\}\|_{\ell^2(\z^2)}^2.
\end{align*}
On the other hand
\begin{align*}
\big\langle G\{c_{k,l}\},\{c_{k,l}\}\big\rangle_{\ell^2(\z^2)}&\geq\frac{1}{\pi^4}\|\{c_{k,l}\}\|_{\ell^2(\z^2)}^2\{I_9-(2|I_1|+4|I_3|+2|I_7|)\}\\
&=\frac{1}{\pi^4}\|\{c_{k,l}\}\|_{\ell^2(\z^2)}^2(14.3661-11.6237)\\
&=\frac{2.7424}{\pi^4}\|\{c_{k,l}\}\|_{\ell^2(\z^2)}^2.
\end{align*}
Thus $\{\t\p_2:k,l\in\z\}$ is a Riesz sequence. 
\end{proof}
We conjecture that $\{\t\p_n:k,l\in\z\}$ is also a Riesz sequence for any twisted $B$-spline of order $n$. In fact, we show that there exists a constant $B>0$ such that
\begin{align*}
\bigg\|\sum_{k,l\in\z}c_{k,l}\t\p_n\bigg\|_{L^2(\R^2)}^2\leq B\|\{c_{k,l}\}\|_{\ell^2(\z^2)}^2\ .
\end{align*}
The proof follows by induction on $n$. By Theorem \ref{c}, the result is true for $n=2$. Assume that the result is true for $n$. We need to show the result for $n+1$. Consider a finite sequence $\{c_{k,l}\}$ in $\ell^2(\z^2)$ and compute
\begin{align*}
\bigg\|\sum_{k,l\in\z}c_{k,l}\t\p_{n+1}\bigg\|_{L^2(\R^2)}^2&=\bigg\|\sum_{k,l\in\z}c_{k,l}W(\t\p_{n+1})\bigg\|_{\b(L^2(\R))}^2\\
&=\bigg\|\sum_{k,l\in\z}c_{k,l}\pi(k,l)W(\p_{n+1})\bigg\|_{\b(L^2(\R))}^2\\
&=\bigg\|\sum_{k,l\in\z}c_{k,l}\pi(k,l)W(\p_n)W(\p_1)\bigg\|_{\b(L^2(\R))}^2\\
&=\bigg\|\Big(\sum_{k,l\in\z}c_{k,l}\pi(k,l)W(\p_n)\Big)W(\p_1)\bigg\|_{\b(L^2(\R))}^2\\
&\leq\bigg\|\sum_{k,l\in\z}c_{k,l}\pi(k,l)W(\p_n)\bigg\|_{\b(L^2(\R))}^2\|W(\p_1)\|_{\b(L^2(\R))}^2\\
&=\bigg\|\sum_{k,l\in\z}c_{k,l}W(\t\p_n)\bigg\|_{\b(L^2(\R))}^2\|\p_1\|_{L^2(\c)}^2\\
&=\bigg\|\sum_{k,l\in\z}c_{k,l}\t\p_n\bigg\|_{L^2(\c)}^2\\
&\leq B\|\{c_{k,l}\}\|_{\ell^2(\z^2)}^2.
\end{align*}
So we obtain an upper bound.\\

The conjecture is mainly for obtaining the lower bound. In the classical case of classical splines $B_n$  on $\R$, the result follows using the fact that $\{T_kB_n:k\in\z\}$ is a {\it Riesz} sequence iff there exist constants $A,B>0$ such that
\begin{align*}
A\leq\sum_{k\in\z}|\widehat{\p}(\xi+k)|^2\leq B,\quad\text{for a.e. $\xi\in\R$}.
\end{align*}
In the case of twisted shift-invariant spaces, we have a similar result (see Theorem $5.1$ in \cite{saswatajmaa}) but under the requirement that ``$\p$ satisfies condition $C$''. However, the splines $\p_n$ do not satisfy this condition for $n\geq 2$. This can be easily checked for $\p_2$. Thus, we leave the conjecture to the interested reader.
\section{Twisted Haar Wavelets}
Recall that $\phi_1(x,y)=\x_{[0,1)}(x)\x_{[0,1)}(y)$ and $\{\t\p_1:k,l\in\z\}$ is an orthonormal system. Our aim is to obtain a wavelet $\psi_1$ by treating $\p_1$ as a scaling function generating a {\it multiresolution analysis} on $\c=\R^2$. We shall call the resulting wavelet $\psi_1$ a twisted Haar wavelet. Using the twisted translation and dilation on $\R^2$, we shall define a $\mra$ in $L^2(\R^2)$. In fact,  a {\it multiresolution analysis} has been studied for the {\it Heisenberg} group $\mathbb{H}$ in \cite{azitamra}. Keeping in mind this {\it multiresolution analysis} for $L^2(\R)$ and $L^2(\mathbb{H})$, we provide the following definition for $L^2(\R^2)$.\\

We recall the definition of twisted translation of a function $\p\in L^2(\R^2)$ from \eqref{11}. The dilation is defined to be the unitary operator on $L^2(\R^2)$ given as follows. For $a>0$, $D_af(x,y):=af(ax,ay)$, for $x,y\in\R$.
\begin{definition}[MRA]\label{a}
We say that a sequence of closed subspaces $\{V_j\}_{j\in\z}$ of $L^2(\R^2)$ forms a \emph{twisted multiresolution analysis} for $L^2(\R^2)$ if the following conditions are satisfied.\\

\begin{enumerate}[(i)]
\item $V_j\subset V_{j+1}$, $\forall\ j\in\z$.\\

\item $\overline{\bigcup\limits_{j\in\z} V_j}=L^2(\R^2)$ and $\bigcap\limits_{j\in\z}V_j=\{0\}$.\\

\item $f\in V_j\quad\Longleftrightarrow\quad D_2f\in V_{j+1},\ \forall\ j\in\z$.\\

\item There exists a function $\p\in V_0$ such that $\{\t\p:k,l\in\z\}$ forms an orthonormal basis for $V_0$. The function $\p$ is called the scaling function associated with the \emph{twisted multiresolution analysis}.
\end{enumerate}
\end{definition}
Let $W_0$ be the orthogonal complement of $V_0$ in $V_1$. As in the classical case, if we can find a $\psi\in W_0$ such that $\{\t\psi:k,l\in\z\}$ forms an orthonormal basis for $W_0$ then by defining $W_j$ to be the orthogonal complement of $V_j$ in $V_{j+1}$, one can show that $L^2(\R^2)=\bigoplus\limits_{j\in\z}W_j$.
\begin{definition}
Let $\psi\in L^2(\R^2)$. If $\{D_{2^j}\t\psi:j,k,l\in\z\}$ is an orthonormal basis for $L^2(\R^2)$, then $\psi$ is called a \emph{twisted wavelet} for $L^2(\R^2)$.
\end{definition}
Now, our aim is to find $\psi$ explicitly for the scaling function $\p_1$ using the \emph{twisted multiresolution analysis}. Towards this end, we prove the following result.
\begin{theorem}\label{m}
Let $\p\in L^2(\R^2)$ such that
\begin{enumerate}
\item[\emph{(i)}] $\supp\p\subset[0,1]\times[0,1]$.
\item[\emph{(ii)}] $\{\t\p:k,l\in\z\}$ forms an orthonormal system.
\item[\emph{(iii)}] $\{V_j,\p:j\in\z\}$ generates a \emph{twisted multiresolution analysis} for $L^2(\R^2)$. 
\end{enumerate}

Let $c_{k,l}:=\big\langle D_{1/2}\p,\t\p\big\rangle$. Define
\begin{align*}
\psi:=\sum_{k,l\in\z}(-1)^{k+l}\ \overline{c_{k,l}}\ D_2T_{(-k+1,l)}^t\p.
\end{align*}
Then $\psi$ is a twisted wavelet for $L^2(\R^2)$. In other words $\{D_{2^j}\t\psi:j,k,l\in\z\}$ is an orthonormal basis for $L^2(\R^2)$. 
\end{theorem}
\begin{proof}
Let $V_0:=\overline{\Span\{\t\p:k,l\in\z\}}$ and $V_j:=\overline{\Span\{D_{2^j}\t\p:k,l\in\z\}}$ for $j\in\z$. As $\p\in V_0$, $\p\in V_1$. Then $D_{1/2}\p\in V_0$ and therefore
\begin{align*}
D_{1/2}\p=\sum_{k,l\in\z}c_{k,l}\t\p.
\end{align*}
Taking the {\it Weyl} transform on both sides, we obtain
\begin{align*}
W(D_{1/2}\p)&=\bigg(\sum_{k,l\in\z}c_{k,l}\pi(k,l)\bigg)W(\p)\\
&=: m_0\ W(\p).
\end{align*}
Define
\begin{align*}
m_1:=\sum_{k,l\in\z}(-1)^{k+l}\ \overline{c_{k,l}}\ \pi(-k+1,l)
\end{align*}
and take $W(D_{1/2}\psi)=m_1W(\p)$. Then,
\begin{align*}
W(D_{1/2}\psi)&=\sum_{k,l\in\z}(-1)^{k+l}\ \overline{c_{k,l}}\ W(T_{(-k+1,l)}^t\p).
\end{align*}
Hence,
\begin{align*}
D_{1/2}\psi=\sum_{k,l\in\z}(-1)^{k+l}\ \overline{c_{k,l}}\ T_{(-k+1,l)}^t\p,
\end{align*}
which in turn implies that
\begin{align*}
\psi=\sum_{k,l\in\z}(-1)^{k+l}\ \overline{c_{k,l}}\ D_2T_{(-k+1,l)}^t\p.
\end{align*}
Clearly $\psi\in V_1$. We need to show that $\p\perp\psi$ and $\{\t\p:k,l\in\z\}\perp\{\t\psi:k,l\in\z\}$. Now,
\begin{align*}
c_{k,l}&=\iint_{\R^2}D_{1/2}\p(x,y)\ \overline{\t\p(x,y)}\ dx\,dy\\
&=\frac{1}{2}\iint_{\R^2}\p\big(\frac{x}{2},\frac{y}{2}\big)\ e^{-\pi i(lx-ky)}\ \overline{\p(x-k,y-l)}\ dx\,dy\\
&=2\iint_{\R^2}\p(x,y)\ e^{-2\pi i(lx-ky)}\ \overline{\p(2x-k,2y-l)}\ dx\,dy.
\end{align*}
As $\supp\p\subset[0,1]\times[0,1]$, we obtain
\begin{align*}
c_{k,l}&=2\int_{x=0}^1\int_{y=0}^1\p(x,y)\ \overline{\p(2x-k,2y-l)}\ e^{-2\pi i(lx-ky)}\ dy\,dx\\
&=\frac{1}{2}\int_{x=-k}^{2-k}\int_{y=-l}^{2-l}\p\big(\frac{x+k}{2},\frac{y+l}{2}\big)\ \overline{\p(x,y)}\ e^{-\pi i(lx-ky)}\ dy\,dx,
\end{align*}
by a change of variables. This shows that $c_{k,l}$ can only contribute for $k=0,1$ and $l=0,1$. Thus, it follows that $\supp\psi\subset[0,1]\times[0,1]$. As $\supp\p\subset[0,1]\times[0,1]$, $\supp\t\p\subset[k,k+1]\times[l,l+1]$. Hence, $\big\langle \psi,\t\p\big\rangle=0,\ \forall\ (k,l)\neq(0,0)$. But,
\begin{align}\label{12}
\big\langle \p,\psi\big\rangle=\sum_{k,l\in\z}(-1)^{k+l}\ c_{k,l}\ \big\langle D_{1/2}\p,T_{(-k+1,l)}^t\p\big\rangle.
\end{align}
As $D_{1/2}\p\in V_0$, $D_{1/2}\p=\sum\limits_{k^\prime,l^\prime\in\z}c_{k^\prime,l^\prime}T_{(k^\prime,l^\prime)}^t\p$. Substitution into \eqref{12} yields
\begin{align*}
\big\langle \p,\psi\big\rangle&=\sum_{k,l,k^\prime,l^\prime\in\z}(-1)^{k+l}\ c_{k,l}c_{k^\prime,l^\prime}\ \big\langle T_{(k^\prime,l^\prime)}^t\p,T_{(-k+1,l)}^t\p\big\rangle\\
&=\sum_{k,l\in\z}(-1)^{k+l}\ c_{k,l}c_{-k+1,l}.
\end{align*} 
Now it is straight-forward to show that the right-hand side of the above equation is zero. Hence it follows that the system $\{\t\p:k,l\in\z\}$ is orthogonal to the system $\{\t\psi:k,l\in\z\}$.\\

Let $W_0$ be the orthogonal complement of $V_0$ in $V_1$. Then $\psi\in W_0$. Further, one can verify that $\{\t\psi:k,l\in\z\}$ is an orthonormal basis for $W_0$. Then it follows from the definition of \emph{twisted multiresolution analysis} that the system $\{D_{2^j}\t\psi:j,k,l\in\z\}$ forms an orthonormal basis for $L^2(\R^2)$. 
\end{proof}

\begin{remark}
Let $\p=\p_1$. We have already shown that $\{\t\p_1:k,l\in\z\}$ is an orthonormal system. Let $V_0:=\overline{\Span\{\t\p_1:k,l\in\z\}}$. However $V_0\not \subset V_1$. In fact, $e^{-\pi iy}\chi_{[1,2]\times [0,1]}\in V_0$ but it does not belong to $V_1$. So we can not make use of Theorem \ref{m}. Hence, we shall define a \textit{nonstationary  multiresolution analysis} in connection with twisted translates in order to obtain a wavelet {\it Riesz} basis for $L^2(\R^2)$. 
\end{remark}
Before giving the definition, let us define $\lambda$-twisted translations.
\begin{definition}
Let $f\in L^2(\R^2)$ and $\lambda\in\R^*:=\R\setminus\{0\}$. For $(k,l)\in\z^2$, the $\lambda$-twisted translation of $f$, denoted by $(\t)^\lambda f$, is defined to be
\begin{align*}
(\t)^\lambda f(x,y):=e^{\pi i\lambda(lx-ky)}f(x-k,y-l),\ \ \ (x,y)\in\R^2.
\end{align*}
\end{definition}
For the above definition and some properties of $\lambda$-twisted translations, we refer to \cite{saswatahouston}. 
\begin{definition}
A sequence of closed subspaces $\{V_j:j\in\z\}$ of $L^2(\R^2)$ is said to generate a nonstationary twisted {\it multiresolution analysis} of $L^2(\R^2)$ if
\begin{enumerate}[(i)]
\item $V_j\subset V_{j+1},\ \forall\ j\in\z$.
\item $\overline{\bigcup\limits_{j\in\z}V_j}=L^2(\R^2)$ and $\bigcap\limits_{j\in\z} V_j=\{0\}$.
\item For $j\in\z$, there exists $\Phi_j\in V_j$ such that $\{D_{2^{-j}}(T^t_{(2^{-j}k,2^{-j}l)})^{2^{2j}}\Phi_j:k,l\in\z\}$ is a {\it Riesz} basis for $V_j$.  
\end{enumerate}
We refer to the functions $\Phi_j$ are $\lambda$-twisted scaling functions at level $j$.
\end{definition}
Then the usual arguments of {\it multiresolution analysis} lead to the fact that if there exists a function $\Psi_j$, for each $j\in\z$, such that 
\be
\{D_{2^{-j}}(T^t_{(2^{-j}k,2^{-j}l)})^{2^{2j}}\Psi_j:k,l\in\z\}
\ee
is a {\it Riesz} basis for $W_j$, then $\{D_{2^{-j}}(T^t_{(2^{-j}k,2^{-j}l)})^{2^{2j}}\Psi_j:j,k,l\in\z\}$ is a \textit{Riesz} basis for $L^2(\R^2)$. The elements of this collection at level $j$ are termed $\lambda$-twisted wavelets.

\begin{example}\label{p}
Let $\p=\p_1$. Let $\Phi_j=N_{2j}$, where $N_j :=2^j\p(2^{j}(\cdot, \cdot))$, $\forall\ j\in\z$. Define $V_j :=\overline{\Span\{D_{2^{-j}}(T^t_{(2^{-j}k,2^{-j}l)})^{2^{2j}}\Phi_j:k,l\in\z\}}$, $j\in\z$. For $j\in\z$, we have
\begin{align*}
D_{2^{-j}}(T^t_{(2^{-j}k,2^{-j}l)})^{2^{2j}}\Phi_j(x,y)=2^je^{\pi i(lx-ky)}\chi_{[k,k+2^{-j}]\times[l,l+2^{-j}]}(x,y).
\end{align*}
Hence $[k,k+2^{-(j+1)}]\times[l,l+2^{-(j+1)}]\subset[k,k+2^{-j}]\times[l,l+2^{-j}]$. This leads to $V_j\subset V_{j+1}$. Let $f\in\bigcap\limits_{j\in\z} V_j$. Then $f(x,y)=c\,e^{\pi i(lx-ky)}\chi_{(-\infty,\infty)\times(-\infty,\infty)}(x,y)$, for some constant $c$. Then $|f(x,y)|=|c|\,\chi_{(-\infty,\infty)\times(-\infty,\infty)}(x,y)$. If $f\in L^2(\R^2)$, then $|f|\in L^2(\R^2)$, which forces $c=0$ and which in turn implies $f=0$, proving that $\bigcap\limits_{j\in\z} V_j=\{0\}$. An element in $\bigcup\limits_{j\in\z}V_j$ can be written as a product of functions $g\cdot u$, where $|g|=1$ and $u$ belongs to the class of piecewise constant functions on $\R^2$. Hence $\bigcup\limits_{j\in\z}V_j$ contains the class of piecewise constant functions on $\R^2$ which implies that $\overline{\bigcup\limits_{j\in\z}V_j}=L^2(\R^2)$. We claim that $\{D_{2^{-j}}(T^t_{(2^{-j}k,2^{-j}l)})^{2^{2j}}\Phi_j:k,l\in\z\}$ is a {\it Riesz} basis for $V_j$ for each fixed $j\in\z$. To this end, let $\{\alpha_{k,l}\}\in\ell^2(\z^2)$. We need to prove that there exist $A,B>0$ such that
\begin{align*}
A\,\|\{\alpha_{k,l}\}\|^2_{\ell^2(\z^2)}\leq \bigg\|\sum_{k,l\in\z}\alpha_{k,l}D_{2^{-j}}(T^t_{(2^{-j}k,2^{-j}l)})^{2^{2j}}\Phi_j\bigg\|^2\leq B\,\|\{\alpha_{k,l}\}\|^2_{\ell^2(\z^2)}.
\end{align*}
But by the choice of our $\Phi_j$, $D_{2^{-j}}(T^t_{(2^{-j}k,2^{-j}l)})^{2^{2j}}\Phi_j=\t N_j$. Now consider
\begin{align*}
S&:=\bigg\|\sum_{k,l\in\z}\alpha_{k,l}D_{2^{-j}}(T^t_{(2^{-j}k,2^{-j}l)})^{2^{2j}}\Phi_j\bigg\|^2\\
&=\bigg\|\sum_{k,l\in\z}\alpha_{k,l}\t N_j\bigg\|^2\\
&=\sum_{k,l,k^\prime,l^\prime\in\z}\alpha_{k,l}\bar{\alpha}_{k^\prime,l^\prime}\langle\t N_j,T^t_{(k^\prime,l^\prime)}N_j\rangle.
\end{align*}
But
\begin{align*}
\langle\t N_j,T^t_{(k^\prime,l^\prime)}N_j\rangle&=2^je^{\pi i(lk^\prime-kl^\prime)}\langle T^t_{(k-k^\prime,l-l^\prime)}N_j,N_j\rangle\\
&=2^je^{\pi i(lk^\prime-kl^\prime)}\int_{x=0}^{2^{-j}}\int_{y=0}^{2^{-j}}T^t_{(k-k^\prime,l-l^\prime)}N_j(x,y)~dydx\\
&=2^je^{\pi i(lk^\prime-kl^\prime)}\int_{x=0}^{2^{-j}}\int_{y=0}^{2^{-j}}e^{\pi i((l-l^\prime)x-(k-k^\prime)y)}N_j(x-(k-k^\prime),y-(l-l^\prime))~dydx\\
&=2^je^{\pi i(lk^\prime-kl^\prime)}\int_{x=-(k-k^\prime)}^{2^{-j}-(k-k^\prime)}\int_{y=-(l-l^\prime)}^{2^{-j}-(l-l^\prime)}e^{\pi i((l-l^\prime)x-(k-k^\prime)y)}N_j(x,y)~dydx,\numberthis \label{23}
\end{align*}
by applying a change of variables. It turns out that $k-k^\prime,l-l^\prime \in \{-2^{-j}+1,-2^{-j}+2,\cdots,0,\cdots,2^{-j}-2,2^{-j}-1\}$. In particular, if $j\geq 0$, then $k=k^\prime$ and $l=l^\prime$, which shows that $\{D_{2^{-j}}(T^t_{(2^{-j}k,2^{-j}l)})^{2^{2j}}\Phi_j:k,l\in\z\}$ is an orthonormal basis for $V_j$ for $j\geq 0$.\\

Now assume that $j<0$. Then
\begin{align*}
S&=\sum_{k,l\in\z}|\alpha_{k,l}|^2+\sum_{k,l\in\z}\sum_{(r,s)\in A_j}e^{\pi i (ks-lr)}\alpha_{k,l}\bar{\alpha}_{k-r,l-s}\langle T^t_{(r,s)}N_j,N_j\rangle\\
&=: \sum_{k,l\in\z}|\alpha_{k,l}|^2+R,
\end{align*}
with $A_j:=\{(r,s)\in\z^2\setminus\{(0,0)\}:r,s \in \{-2^{-j}+1,-2^{-j}+2,\cdots,0,\cdots,2^{-j}-2,2^{-j}-1\}\}$. Then $A_j$ is the disjoint union of $A^{(1)}_j$ and $A^{(2)}_j$ where $A^{(1)}_j:=\{(0,s):s\in \{1,2,\cdots,2^{-j}-1\}\}\cup\{(r,s)\in\z^2:r\in\{1,2,\cdots,2^{-j}-1\},s\in \{-2^{-j}+1,-2^{-j}+2,\cdots,0,\cdots,2^{-j}-2,2^{-j}-1\}\}$ and $A^{(2)}_j=\{(r,s)\in\z^2:(-r,-s)\in A^{(1)}_j\}$. We notice that $\#(A^{(1)}_j)=\#(A^{(2)}_j)=2^{-j+1}(2^{-j}-1)$ and hence $\#(A_j)=2^{-j+2}(2^{-j}-1)$. (Here, $\#$ denotes the cardinality of a set.) Now, $R$, as defined above, can be expressed as
\begin{align*}
R&=\sum_{k,l\in\z}\bigg[\sum_{(r,s)\in A^{(1)}_j}e^{\pi i (ks-lr)}\alpha_{k,l}\bar{\alpha}_{k-r,l-s}\langle T^t_{(r,s)}N_j,N_j\rangle+\sum_{(r,s)\in A^{(2)}_j}e^{\pi i (ks-lr)}\alpha_{k,l}\bar{\alpha}_{k-r,l-s}\langle T^t_{(r,s)}N_j,N_j\rangle\bigg]\\
&=: R_1+R_2.
\end{align*}
We observe that $R_2=\bar{R}_1$ from which it follows that
\begin{align*}
|R|&\leq 2\sum_{(r,s)\in A^{(1)}_j}|\langle T^t_{(r,s)}N_j,N_j\rangle|\sum_{k,l\in\z}|\alpha_{k,l}||\alpha_{k-r,l-s}|\\
&\leq 2\|\{\alpha_{k,l}\}\|^2_{\ell^2(\z^2)}\sum_{(r,s)\in A^{(1)}_j}|\langle T^t_{(r,s)}N_j,N_j\rangle|. \numberthis \label{20} 
\end{align*}
If at least one of $r$ and $s$ is equal to zero, then \eqref{23} implies that $\langle T^t_{(r,s)}N_j,N_j\rangle=0$. For $r\neq 0$ and $s\neq 0$, a straightforward computation using \eqref{23} leads to
\begin{align}\label{21}
\langle T^t_{(r,s)}N_j,N_j\rangle&=
\begin{cases}
\dfrac{2^{2j+1}}{\pi^2rs}(1-\cos(\pi rs)), & r,s\in \{1,2,\cdots,2^{-j}-1\}\\ \\
-\dfrac{2^{2j+1}}{\pi^2rs}(1-\cos(\pi rs)), & r\in\{1,2,\cdots,2^{-j}-1\},\ s\in \{-2^{-j}+1,-2^{-j}+2,\cdots,-1\}.
\end{cases}
\end{align}
In either of these cases,
\begin{align*}
|\langle T^t_{(r,s)}N_j,N_j\rangle|&=\frac{2^{2j+1}}{\pi^2rs}(1-\cos(\pi rs))\\
&=\frac{2^{2j+2}}{\pi^2rs}\sin^2(\frac{\pi rs}{2})\\
&=\frac{2^{2j+1}}{\pi}\sinc(\frac{\pi rs}{2})\sin(\frac{\pi rs}{2})\\
&\leq\frac{2^{2j+1}}{\pi}.
\end{align*}
Furthermore, \eqref{21} implies that $\langle T^t_{(r,s)}N_j,N_j\rangle=0$ if at least any one of $r$ and $s$ is even. Therefore, let $B^{(1)}_j$ be the subset of $A^{(1)}_j$ for which $\langle T^t_{(r,s)}N_j,N_j\rangle\neq 0$. Then, \eqref{20} reduces to
\begin{align*}
|R|&\leq 2\|\{\alpha_{k,l}\}\|^2_{\ell^2(\z^2)}\sum_{(r,s)\in B^{(1)}_j}|\langle T^t_{(r,s)}N_j,N_j\rangle|\\
&\leq \Big(\frac{2^{2j+2}}{\pi}\#(B^{(1)}_j)\Big)\|\{\alpha_{k,l}\}\|^2_{\ell^2(\z^2)}.\numberthis \label{22}
\end{align*}
But $B^{(1)}_j$ is the set of pairs $(r,s)$, where both $r$ and $s$ are odd. Thus $\#(B^{(1)}_j)=2^{-2j-1}$. Hence, \eqref{22} gives $|R|\leq \frac{2}{\pi}\|\{\alpha_{k,l}\}\|^2_{\ell^2(\z^2)}$, from which it follows that
\begin{align*}
\Big(1-\frac{2}{\pi}\Big)\|\{\alpha_{k,l}\}\|^2_{\ell^2(\z^2)}\leq S\leq \Big(1+\frac{2}{\pi}\Big)\|\{\alpha_{k,l}\}\|^2_{\ell^2(\z^2)},
\end{align*}
proving our assertion. Thus, if one can find a function $\Psi_j$, for each $j\in\z$ in $L^2(\R^2)$, such that 
\[
\{D_{2^{-j}}(T^t_{(2^{-j}k,2^{-j}l)})^{2^{2j}}\Psi_j:k,l\in\z\}
\]
is a {\it Riesz} basis for $W_j\ (V_j\oplus W_j=V_{j+1})$, then
\[
\{D_{2^{-j}}(T^t_{(2^{-j}k,2^{-j}l)})^{2^{2j}}\Psi_j:j,k,l\in\z\}
\]
is a \textit{Riesz} basis for $L^2(\R^2)$. We leave it as an open question to the interested reader.
\end{example}

In Figure \ref{fig2} below, we have depicted $D_{2^{-j}}(T^t_{(2^{-j}k,2^{-j}l)})^{2^{2j}}\Phi_j\in V_j$ for two triples $(j,k,l)$.\\

\begin{figure}[h!]
\begin{center}
\includegraphics[width=5cm, height= 4cm]{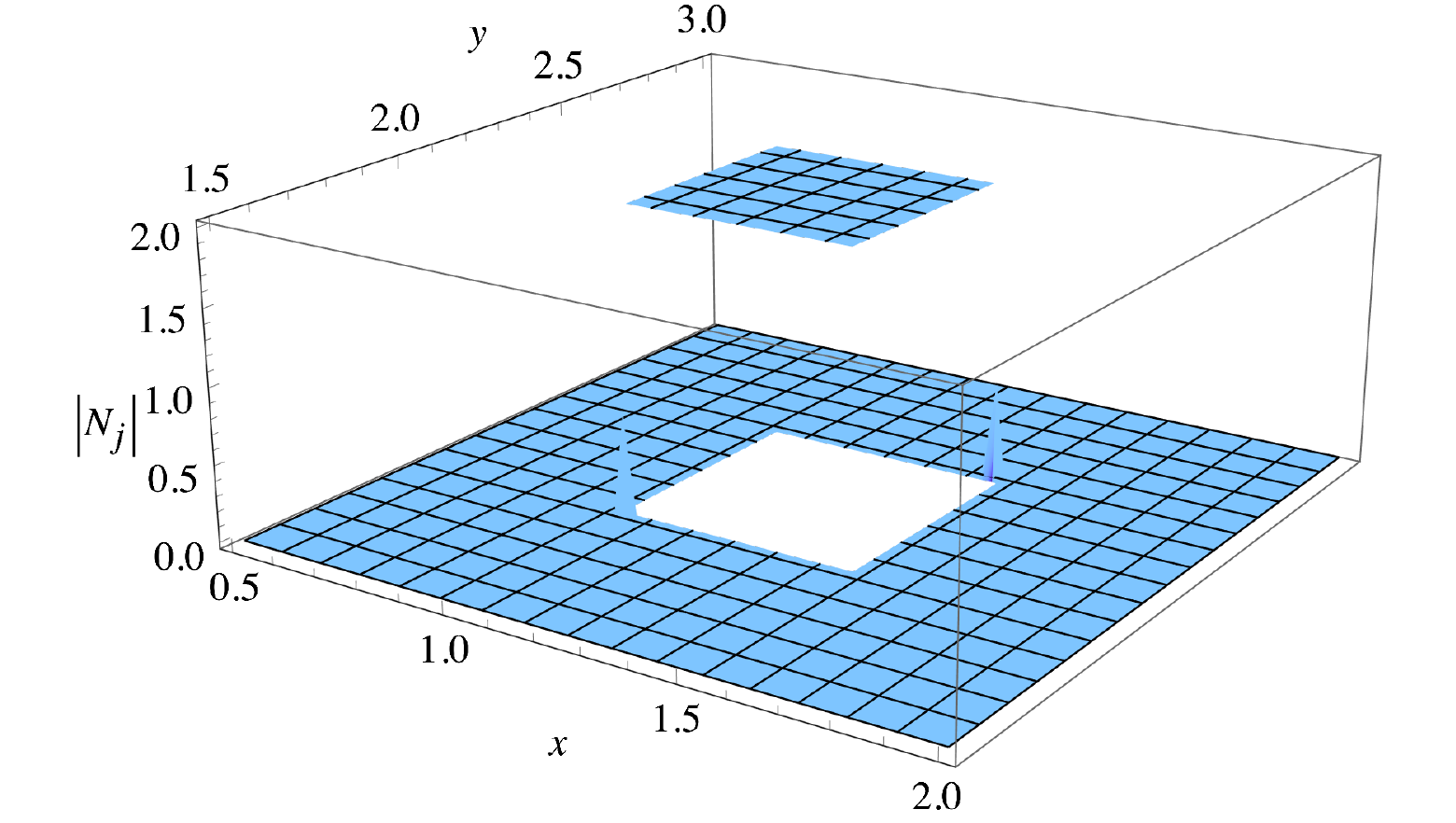}\hspace{2cm}
\includegraphics[width=5cm, height= 4cm]{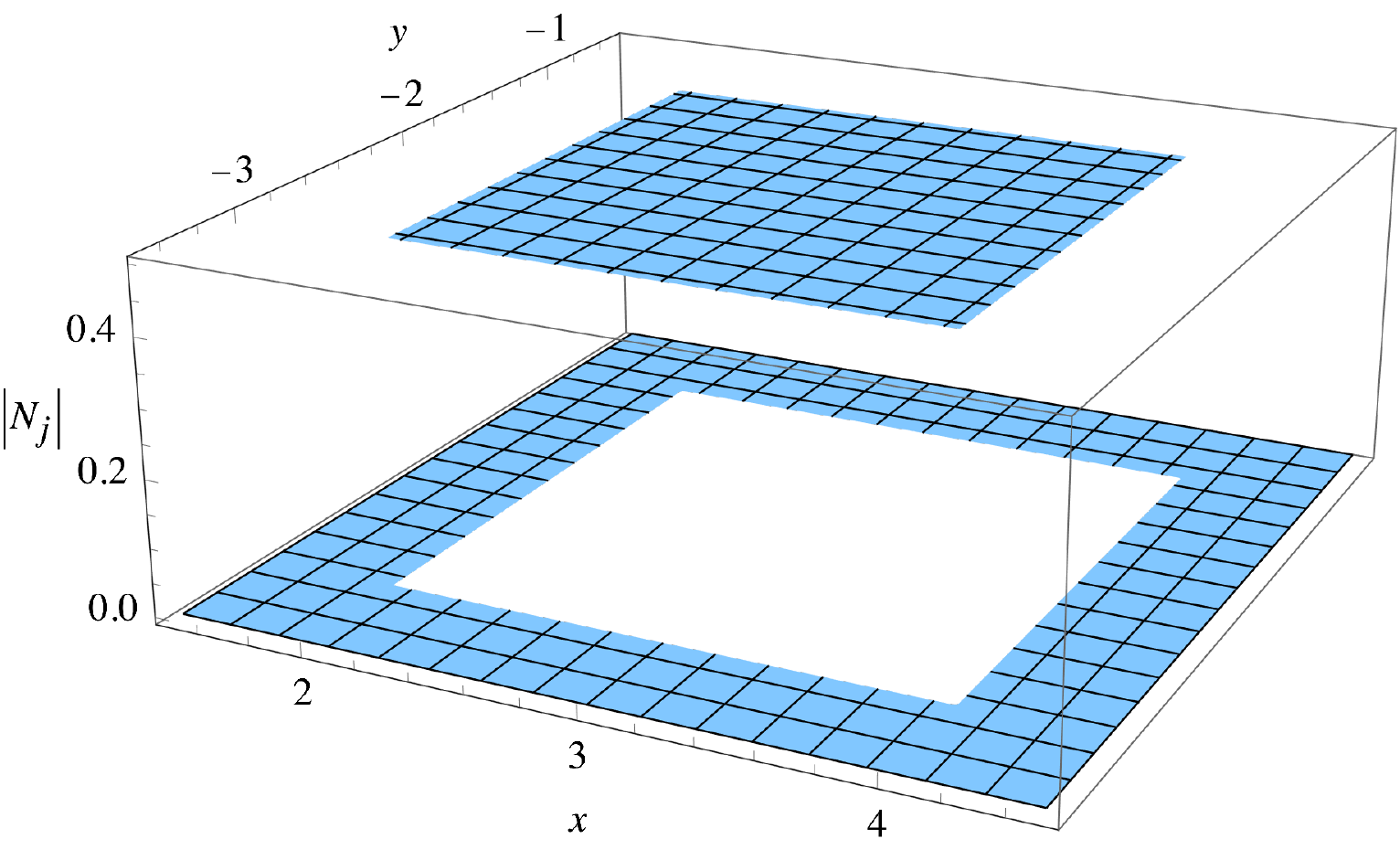}
\end{center}
\end{figure}
\begin{figure}[h!]
\begin{center}
\includegraphics[width=5cm, height= 4cm]{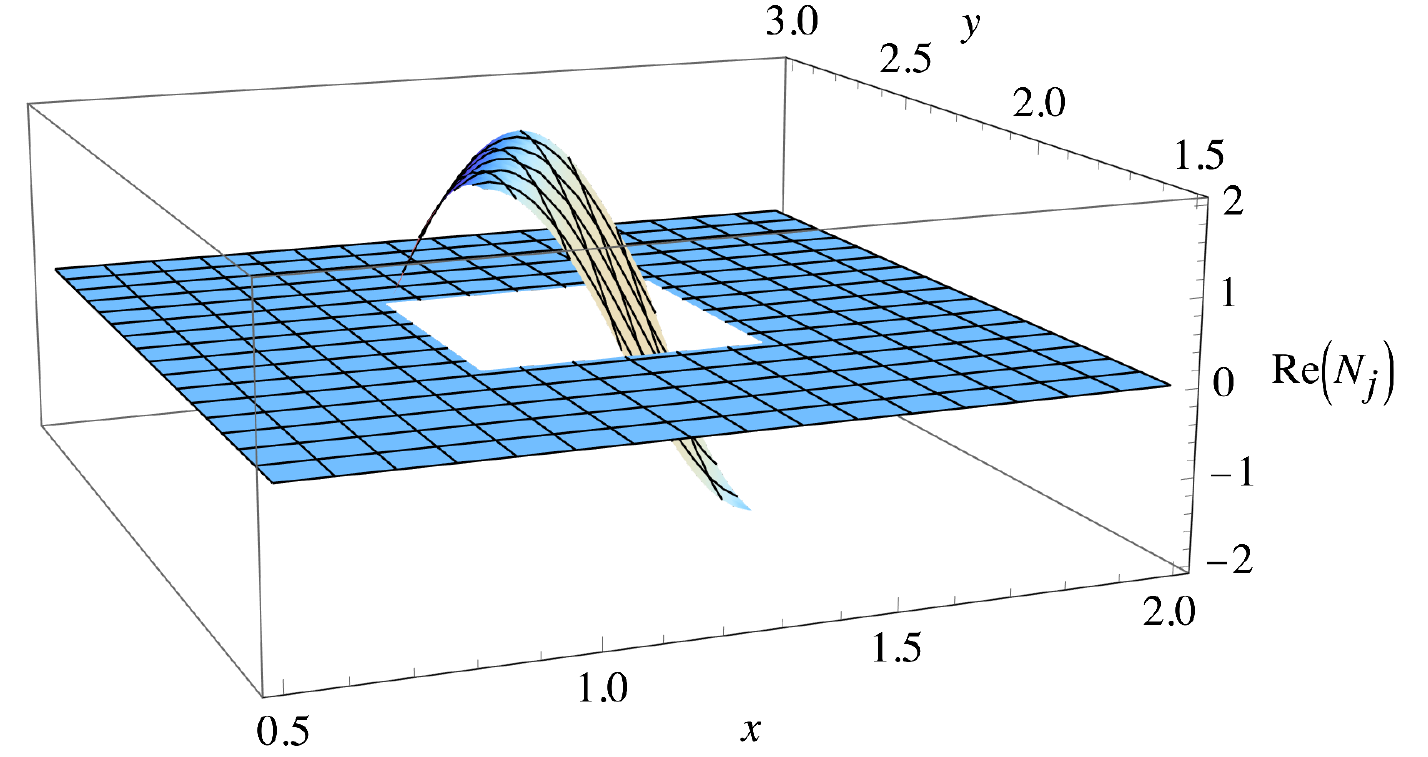}\hspace{2cm}
\includegraphics[width=5cm, height= 4cm]{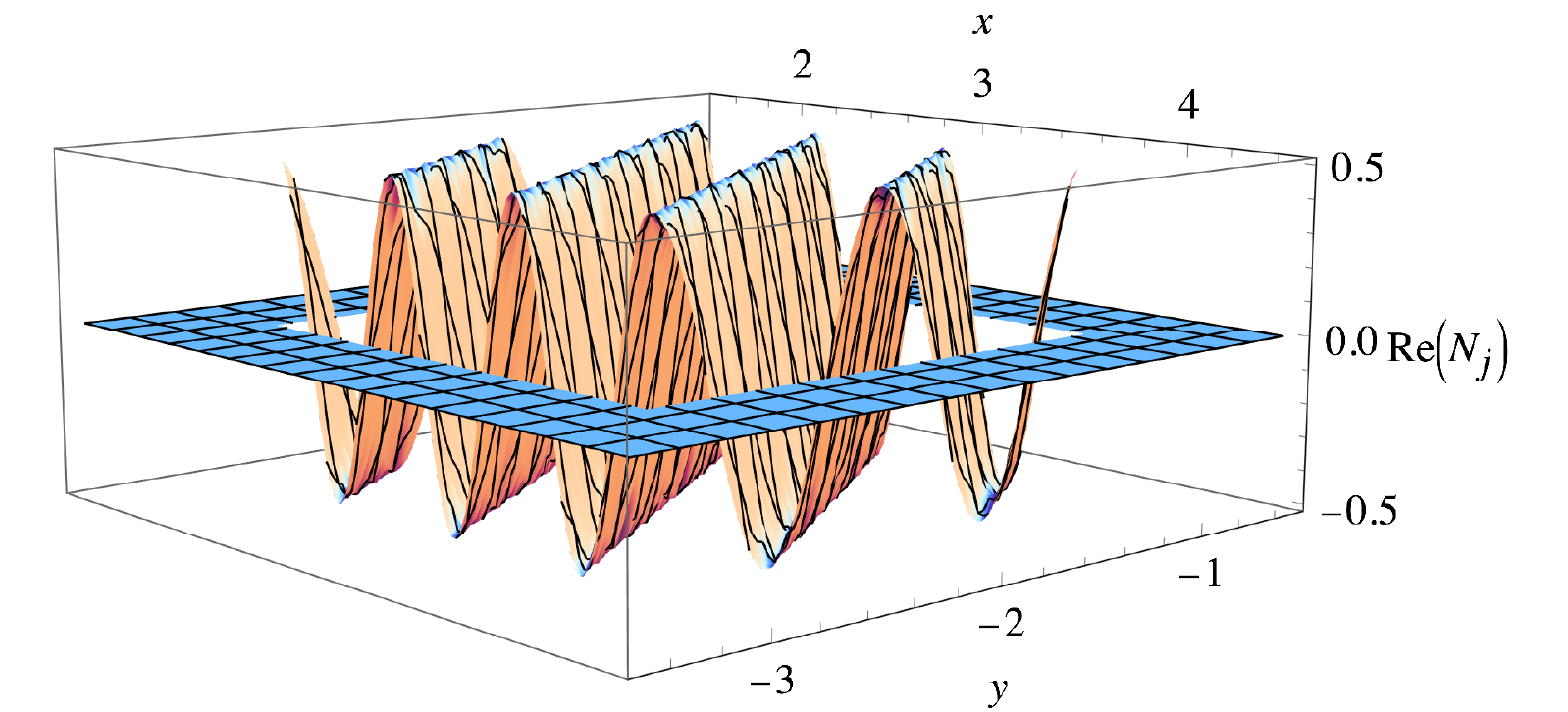}
\end{center}
\end{figure}
\begin{figure}[h!]
\begin{center}
\includegraphics[width=5cm, height= 4cm]{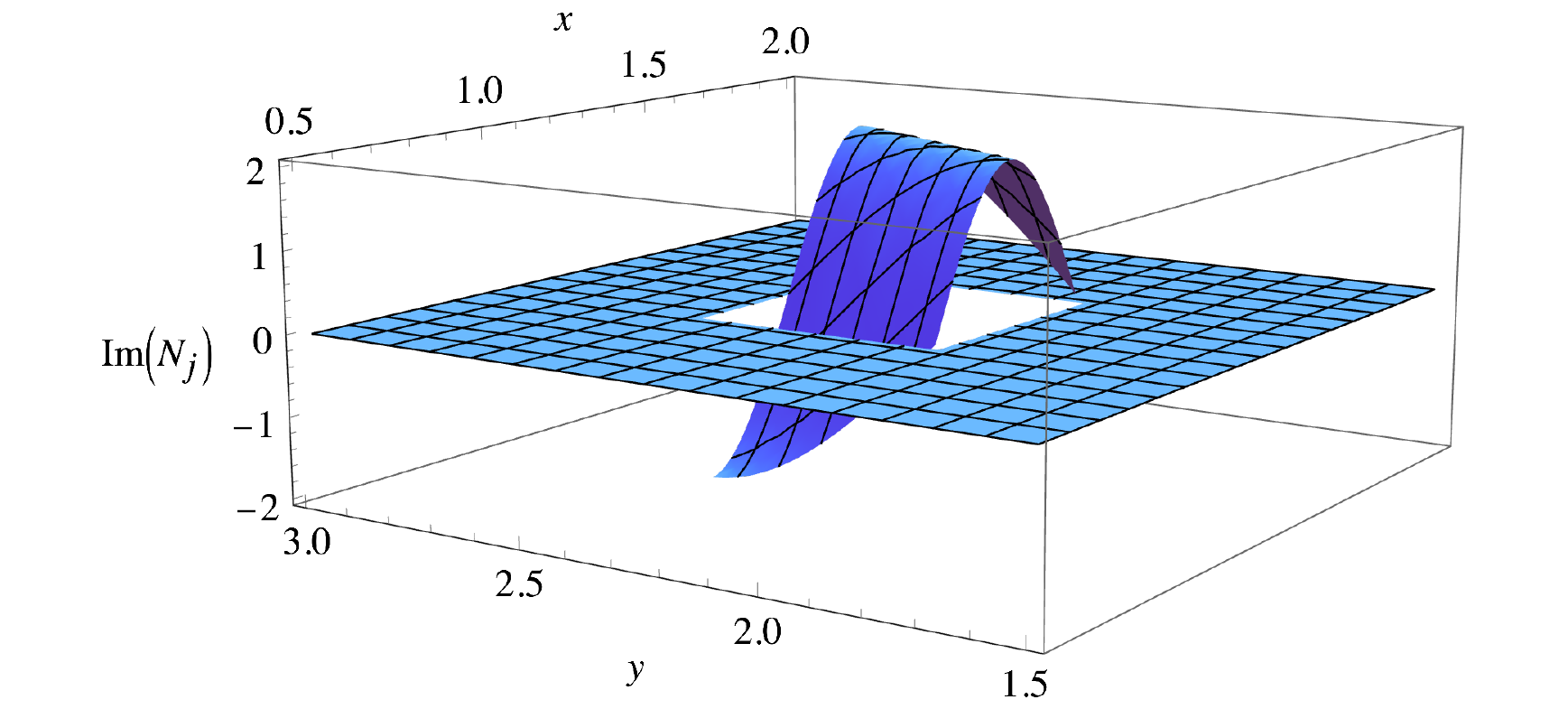}\hspace{2cm}
\includegraphics[width=5cm, height= 4cm]{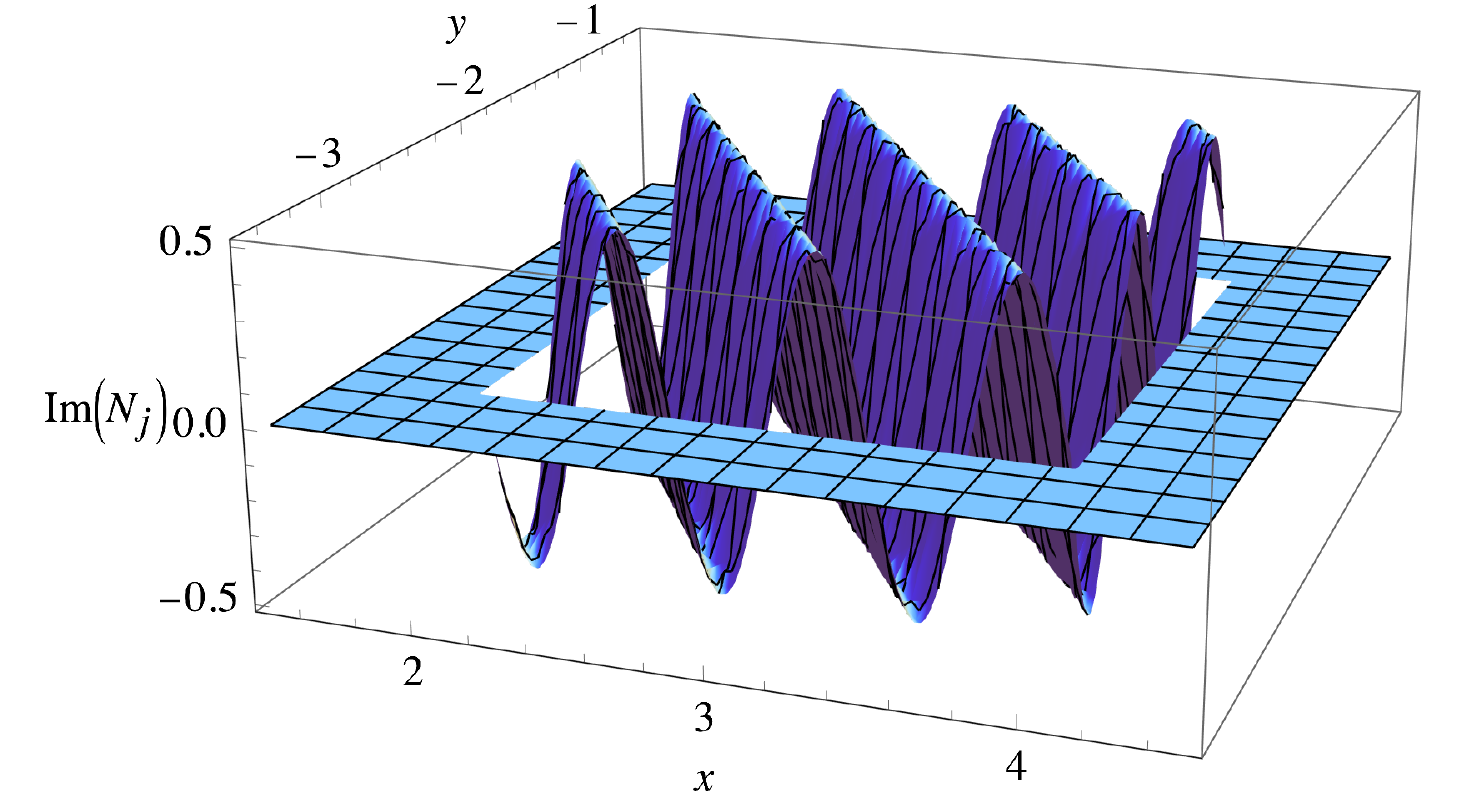}
\caption{For $(j,k,l) = (1,1,2)$ (left) and $(j,k,l) = (-1,2,-3)$ (right), the modulus (upper row), the real part (middle row), and the imaginary part (lower row) of $D_{2^{-j}}(T^t_{(2^{-j}k,2^{-j}l)})^{2^{2j}}\Phi_j$ are shown.}\label{fig2}
\end{center}
\end{figure}

\section*{Acknowledgement}                                                                                                                                                                                                                                                                                                                                                                                                                                                                                                                                                                                                                                                                                                                                                                                                                                                                                                                                                                                                                                                                                                                                                                                                                                                                                                                                                                                                                                                                                                                                                                                                                                                                                                                                                                                                                                                                                                                                                                                                                                                                                                                                                                                                                                                                                                                                                                                                                                                                                                                                                                                                                                                                                                                                                                                                                                                                                                                                                                                                                                                                                                                                                                                                                                                                                                                                                                                                                                                                                                                                                                                                                                                                                                                                                                                                                                                                                                                                                                                                                                                                                                                                                                                                                                     
We thank the referee for carefully reading the manuscript and giving us valuable suggestions. One of the authors (R.R) was partially supported by the Visiting Professor Program funded by the Bavarian State Ministry for Science, Research and the Arts, TUM International Center, for a stay at the Technical University Munich, Germany, for the period October to December 2019. She sincerely thanks Professor Massimo Fornasier and Professor Peter Massopust, TUM, for their kind invitation and excellent hospitality for the entire period of her visit from October 2019 to September  2020.  

\bibliographystyle{amsplain}
\bibliography{splines}

\providecommand{\bysame}{\leavevmode\hbox to3em{\hrulefill}\thinspace}
\providecommand{\MR}{\relax\ifhmode\unskip\space\fi MR }
\providecommand{\MRhref}[2]{%
  \href{http://www.ams.org/mathscinet-getitem?mr=#1}{#2}
}
\providecommand{\href}[2]{#2}
\begin{thebibliography}{10}

\bibitem{aratijmpa}
S.~Arati and R.~Radha, \emph{Orthonormality of wavelet system on the
  {H}eisenberg group}, J. Math. Pures Appl. (9) \textbf{131} (2019), 171--192.
  \MR{4021173}

\bibitem{bastin}
F.~Bastin and L.~Simons, \emph{{About nonstationary multiresolution analysis
  and wavelets}}, {Results Math.} \textbf{{63}} ({2013}), no.~{1-2},
  {485--500}. \MR{{3009700}}

\bibitem{BCT}
D.C.~Chang C.~Berenstein and J.~Tie, \emph{{Laguerre Calculus and its
  Applications in the Heisenberg Groups}}, {IP {S}eries in {A}dvanced
  {M}athematics}, vol.~{22}, {American Mathematical Society}, {2001}.

\bibitem{C}
O.~Christensen, \emph{An introduction to frames and {R}iesz bases}, second ed.,
  Applied and Numerical Harmonic Analysis, Birkh\"{a}user/Springer, [Cham],
  2016. \MR{3495345}

\bibitem{Ch}
C.K. Chui, \emph{An introduction to wavelets}, Wavelet Analysis and its
  Applications, vol.~1, Academic Press, Inc., Boston, MA, 1992. \MR{1150048}

\bibitem{CS}
H.B. Curry and I.J. Schoenberg, \emph{On {P}\'{o}lya frequency functions. {IV}.
  {T}he fundamental spline functions and their limits}, J. Analyse Math.
  \textbf{17} (1966), 71--107. \MR{218800}

\bibitem{daubechies}
I.~Daubechies, \emph{Ten lectures on wavelets}, CBMS-NSF Regional Conference
  Series in Applied Mathematics, vol.~61, Society for Industrial and Applied
  Mathematics (SIAM), Philadelphia, PA, 1992. \MR{1162107}

\bibitem{boor}
C.~de~Boor, R.~DeVore, and A.~Ron, \emph{{On the construction of multivariate
  (pre)wavelets}}, {Constr. Approx.} \textbf{{9}} ({1993}), no.~{2-3},
  {123--166}. \MR{{1215767}}

\bibitem{follandphase}
G.B. Folland, \emph{Harmonic analysis in phase space}, Annals of Mathematics
  Studies, vol. 122, Princeton University Press, Princeton, NJ, 1989.
  \MR{983366}

\bibitem{leb}
N.N. Lebedev, \emph{Special functions and their applications}, Prenitce Hall,
  Inc., 1965.

\bibitem{mallat}
S.~G. Mallat, \emph{Multiresolution approximations and wavelet orthonormal
  bases of {$L^2({\bf R})$}}, Trans. Amer. Math. Soc. \textbf{315} (1989),
  no.~1, 69--87. \MR{1008470}

\bibitem{peter}
P.~Massopust, \emph{Interpolation and approximation with splines and fractals},
  Oxford University Press, Oxford, 2010. \MR{2723033}

\bibitem{azitamra}
A.~Mayeli, \emph{Shannon multiresolution analysis on the {H}eisenberg group},
  J. Math. Anal. Appl. \textbf{348} (2008), no.~2, 671--684. \MR{2445768}

\bibitem{meyer}
Y.~Meyer, \emph{Ondelettes et fonctions splines}, S\'{e}minaire sur les
  \'{e}quations aux d\'{e}riv\'{e}es partielles 1986--1987, \'{E}cole
  Polytech., Palaiseau, 1987, pp.~Exp. No. VI, 18. \MR{920024}

\bibitem{saswatajmaa}
R.~Radha and S.~Adhikari, \emph{Frames and {R}iesz bases of twisted
  shift-invariant spaces in {$L^2(\Bbb{R}^{2n})$}}, J. Math. Anal. Appl.
  \textbf{434} (2016), no.~2, 1442--1461. \MR{3415732}

\bibitem{saswatacollec}
\bysame, \emph{Left translates of a square integrable function on the
  {H}eisenberg group}, Collect. Math. \textbf{71} (2020), no.~2, 239--262.
  \MR{4083644}

\bibitem{saswatahouston}
R.~Radha and S.~Adhikari, \emph{Shift-invariant spaces with countably many
  mutually orthogonal generators on the {Heisenberg} group}, Houston J. Math.
  (2020), no.~46, 435--463.

\bibitem{thangavelu}
S.~Thangavelu, \emph{Harmonic analysis on the {H}eisenberg group}, Progress in
  Mathematics, vol. 159, Birkh\"{a}user Boston, Inc., Boston, MA, 1998.
  \MR{1633042}

\bibitem{blu}
C.~Vonesch, Th. Blu, and M.~Unser, \emph{{Generalized {D}aubechies wavelet
  families}}, {IEEE Trans. Signal Process.} \textbf{{55}} ({2007}), no.~{9},
  {4415--4429}. \MR{{2464454}}

\end{thebibliography}
\end{document}